\documentclass[10pt]{amsart}

\usepackage{microtype}
\usepackage{amsmath, amsthm, amscd, amsfonts, amssymb, graphicx, color}

\numberwithin{equation}{section}
\numberwithin{figure}{section}
\theoremstyle{plain}
\newtheorem{thm}{\protect\theoremname}[section]
\newtheorem{lem}[thm]{\protect\lemmaname}
\newtheorem{prop}[thm]{\protect\propositionname}
\newtheorem{example}[thm]{Example}
\newtheorem{rem}[thm]{Remark}
\newtheorem{defn}[thm]{Definition}
\newtheorem{cor}[thm]{Corollary}

\providecommand{\lemmaname}{Lemma}
\providecommand{\propositionname}{Proposition}
\providecommand{\theoremname}{Theorem}

\def\quot{/\!\!/}

\def\ker{\mathsf{Ker}}
\def\hom{\mathsf{Hom}}
\def\im{\mathsf{Im}}

\def\X{\mathfrak{X}}
\DeclareMathOperator{\GL}{\mathrm{GL}}
\DeclareMathOperator{\SL}{\mathrm{SL}}

\DeclareMathOperator{\U}{\mathrm{U}}
\def\C{\mathbb{C}}
\def\Z{\mathbb{Z}}

\def\N{\mathbb{N}}
\def\n{\mathbf{n}}
\def\id{\mathbf{1}}

\makeatletter
\@namedef{subjclassname@2020}{\textup{2020} Mathematics Subject Classification}
\makeatother

\begin{document}

\title[Generalized Torus Knot Groups \& Character Varieties.]{Character Varieties of Generalized Torus Knot Groups.}

\author[C. Florentino]{Carlos Florentino}

\address{Dep. Matematica and CEMS.UL - Center for Mathematical Studies, Fac. Ci\^{e}ncias, Univ. Lisboa, Edf. C6, Campo Grande, 1749-016 Lisboa, Portugal.}

\email{caflorentino@fc.ul.pt}

\author[S. Lawton]{Sean Lawton}

\address{Department of Mathematical Sciences, George Mason University, 4400 
University Drive, Fairfax, Virginia 22030, USA}

\email{slawton3@gmu.edu}

\date{\today}

\keywords{Character varieties, generalized torus knot/link groups, flawed groups}

\subjclass[2020]{Primary 14M35, 20F38, 14P25; Secondary 14L30, 14D20, 22E46}



\begin{abstract}
Given $\mathbf{n}=(n_{1},\ldots,n_{r})\in\N^r$, let
$\Gamma_{\mathbf{n}}$ be a group presentable as $$\left\langle \gamma_{1},\ldots,\gamma_{r}\:|\:\gamma_{1}^{n_{1}}=\gamma_{2}^{n_{2}}=\cdots=\gamma_{r}^{n_{r}}\right\rangle. $$  If $\gcd(n_i,n_j)=1$ for all $i\not=j$, we say $\Gamma_{\mathbf{n}}$ is a {\it generalized torus knot group} and otherwise say it is a {\it generalized torus link group}. This definition includes torus knot and link groups ($r=2$), that is, fundamental groups of the complement of a torus knot or link in $S^{3}$.  Let $G$ be a connected complex reductive affine algebraic group.  We show that the $G$-character varieties of generalized torus knot groups are path-connected. We then count the number of irreducible components of the $\SL(2,\C)$-character varieties of $\Gamma_\n$ when $n_i$ is odd for all $i$.
\end{abstract}

\maketitle

\section{Introduction}
\subsection{Background}
Let $G$ be a connected complex reductive affine algebraic group (a reductive $\C$-group), and let it act by conjugation on the space of homomorphisms $\hom(\Gamma, G)$ where $\Gamma$ is a finitely presented group generated by $r$ elements.  We may consider $\hom(\Gamma, G)$ a subspace of $G^r$ by identifying homomorphisms with their evaluations at generators of $\Gamma$. Let $\hom^*(\Gamma, G)$ be the subspace of $\hom(\Gamma, G)$ consisting of homomorphisms with closed conjugation orbits. The quotient space by the conjugation action $\X_\Gamma(G):=\hom^*(\Gamma, G)/G$ is called the {\it $G$-character variety of $\Gamma$}.  The space $\X_\Gamma(G)$ is homeomorphic to the (affine) Geometric Invariant Theory (GIT) quotient $\hom(\Gamma, G)\quot G$ equipped with the analytic topology by \cite[Theorem 2.1]{FlLa4}, and so inherits a natural algebraic structure.  Consequently, it is Hausdorff, although highly singular (\cite{FlLa2,FLR,GLR}).  On the other hand, $\X_\Gamma(G)$ is homotopic to the non-Hausdorff conjugation quotient $\hom(\Gamma, G)/G$ by \cite[Proposition 3.4]{FLR}.  

The fundamental group of the complement in $S^3$ of a torus knot, a knot that lies on the surface of an embedded standard torus in $\mathbb{R}^3$,  is called a torus knot group.  Torus knot groups admit the presentation:
\[
\Gamma=\left\langle \gamma_{1},\gamma_{2}\:|\:\gamma_{1}^{n_{1}}=\gamma_{2}^{n_{2}}\right\rangle ,
\]
for some coprime integers $n_{1},n_{2}\in\mathbb{N}$. When $n_{1}$
and $n_{2}$ are not necessarily coprime, these will be torus link
groups, as these are fundamental groups of the complement of links lying on an embedded torus in $S^3$ (see \cite{Lickorish}, for example).

\subsection{Main results}

Fix $r\in\mathbb{N}$ and consider, more generally,
groups with the following presentation:
\[
\Gamma_{\mathbf{n}}:=\left\langle \gamma_{1},\ldots,\gamma_{r}\:|\:\gamma_{1}^{n_{1}}=\gamma_{2}^{n_{2}}=\cdots=\gamma_{r}^{n_{r}}\right\rangle ,
\]
for a fixed $\mathbf{n}=(n_{1},\ldots,n_{r})\in \N^r$.  We call such groups {\it generalized torus link groups}.  If $\gcd(n_i,n_j)=1$ for all $i\not=j$, then we say $\Gamma_{\mathbf{n}}$ is a {\it generalized torus knot group}.

Motivated by examples in \cite{munozsl2, munozsu2, ollersl2link},  in this paper we study the geometry and topology of character varieties of generalized torus knot groups.

Our main theorem, generalizing results in the above papers, is:

\begin{thm}
Let $\Gamma_\n$ be a generalized torus knot group with $r\geq 2$ and $G$ a reductive $\C$-group.  Then $\hom(\Gamma_\n,G)$ and $\X_{\Gamma_\n}(G)$ are path-connected.
\end{thm}

To prove this theorem, we show that the abelian locus in $\X_{\Gamma_\n}(G)$ and $\hom(\Gamma_\n,G)$ is path-connected (Theorem \ref{abelian-thm}) and then we show that there exists a path from any representation of $\Gamma_\n$ to an abelian one.

To construct such a path we need a detailed analysis of the case $G=\GL(m,\C)$.  In this case, we prove there exists a natural stratification for all torus link groups and describe it geometrically in addition to giving an explicit linear algebraic condition for a tuple of matrices to be a representation of $\Gamma_\n$ (Propositions \ref{prop:diagonalization} and \ref{prop:solutions}).

More generally, we show with Theorem \ref{flawed-thm} that $\hom^*(\Gamma_\n,G)$ fits into a natural ``exact sequence" with $\hom^*(\Z,G)$ and $\hom^*(*_{i=1}^r\Z_{n_i},G)$. Therefore, every point in $\X_{\Gamma_\n}(G)$ is determined by the simpler data of $G$-representations of free products of cyclic groups.

Lastly, in the cases of $G=\GL(2,\C)$ and $G=\SL(2,\C)$, we describe the combinatorics of the locus of irreducible representations. These results generalize the component counting results in \cite{munozsl2, ollersl2link}.  Precisely, here is the main theorem of Section \ref{sec:irrcomp} (Theorem \ref{thm:components}): 

\begin{thm}
Suppose all $n_{i}$ in $\n$ are odd. Then, the number of irreducible components
of $\hom^{irr}(\Gamma_{\n},\SL(2,\C))$, and consequently of $\X^{irr}_{\Gamma_\n}(\SL(2,\C))$, is:
\[
r-2-\sum_{i=1}^{r}n_{i}+\frac{1}{2^{r-1}}\prod_{i=1}^{r}(n_{i}+1).
\]
\end{thm}

\subsection{Motivation}
There are two main reasons for this study.  The first is the search for {\it flawed} groups (see Remark \ref{rem:flawed} for the definition).  It is conjectured that all generalized torus link groups are flawed \cite{FLflawed}.  This conjecture, philosophically, makes sense since generalized torus link groups have properties similar to both abelian and free groups which are both flawed \cite{FlLa1, FlLa4}.  On the other hand, the results in \cite{munozsl2, munozsu2, munozsl3} show the conjecture is true in some cases.  So to ascertain whether this conjecture is true in general, a systematic study of these groups and their character varieties is necessary, the present paper being the first such study.  The second motivation is to determine what properties from torus links and knots are important in the study of their representations and what properties are simply group theoretic.  For example, the results in Section \ref{sec:la} show that many of the technical facts about these moduli spaces are entirely group-theoretically determined and have little to do with that fact that the groups (when $r=2$) arise from link complements.

\subsection*{Acknowledgments}
We acknowledge support from U.S. National Science Foundation grants: DMS 1107452, 1107263, 1107367 ``RNMS: GEometric structures And Representation varieties" (the GEAR Network), and thank the Simons Center for Geometry and Physics, Stony Brook University for its hospitality where some the research for this paper was performed. Lawton was supported by a collaboration grant from the Simons Foundation (\#578286).  Florentino was supported by CMAFcIO (University of Lisbon), and the project UIDB/04561/2020, FCT Portugal.  Lastly, we thank an anonymous referee for helpful suggestions.

\section{Generalized Torus Link Groups and their Representations}

Let $\Gamma$ be a finitely generated group and $G$ a reductive $\C$-group. As $G$ admits a faithful representation, we will often refer to homomorphisms $\rho\in \hom(\Gamma, G)$ as $G$-representations or simply representations.  As in \cite{Sikora-CharVar}, we say that a given representation $\rho\in\hom(\Gamma,G)$ is \emph{polystable} if the conjugation orbit $G\cdot\rho$ is closed (either in the Zariski or in the analytic topology, see \cite[Page 5]{FlLa4}). An element $g\in G$ is {\it semisimple} if it is polystable when thought of as an element in $\hom(\Z,G)$.  An element $g\in G$ is {\it regular} if its centralizer $Z_G(g)$ has minimal dimension.  For example, semisimple elements in $\GL(n,\C)$ are diagonalizable, but regular semisimple elements have distinct eigenvalues.

We say that $\rho$ is {\it irreducible} if $\rho(\Gamma)$ is not contained in any proper parabolic subgroup in $G$.  If for every proper parabolic $P\subset G$ such that $\rho(\Gamma)\subset P$, there is a Levi subgroup $L\subset P$ such that $\rho(\Gamma)\subset L$ then $\rho$ is {\it completely-reducible}.  So irreducible implies completely-reducible.  Moreover, \cite[Theorem 30]{Sikora-CharVar} shows that polystability and complete-reducibility coincide. 

Fix a presentation of $\Gamma$ with $r$ generators $\{\gamma_1,...,\gamma_r\}$.  Then there is an injection $\hom(\Gamma,G)\hookrightarrow G^r$ by $\rho\mapsto (\rho(\gamma_1),...,\rho(\gamma_r))$ which induces a topology on $\hom(\Gamma,G)$ independent of the presentation of $\Gamma$ (up to homeomorphism).  Additionally, it gives $\hom(\Gamma,G)$ the structure of an affine algebraic set since $G$ is an affine variety.  Precisely, the relations of $\Gamma$ determine polynomials that cut out $\hom(\Gamma,G)$ from $G^r$. 

\subsection{Abelian Representations of Torus Knot Groups}

Let  $\n=(n_1,...,n_r)\in\N^r$. We now show that the abelianization of the group 
\[
\Gamma_{\mathbf{n}}=\left\langle \gamma_{1},...,\gamma_{r}\mid\gamma_{i}^{n_{i}}=\gamma_{j}^{n_{j}}\text{ for all } i,j\right\rangle 
\]
is free of rank one if and only if  $\Gamma_{\n}$ is a generalized torus knot group,
that is, $\gcd(n_i,n_j)=1$ for all $i\not=j$. 

To simplify notation, we write the abelianization of $\Gamma_{\n}$ additively as:
\[
\Gamma_{\mathbf{n}}^{\text{ab}}:= \Gamma_{\n}/[\Gamma_\n,\Gamma_\n] = \left\langle \gamma_{1},...,\gamma_{r}\mid n_{i}\gamma_{i}=n_{j}\gamma_{j}\text{ for all } i,j\right\rangle ,
\]
using the same generators $\gamma_i$. As such, $\Gamma_{\mathbf{n}}^{\text{ab}}$ is the cokernel of the linear map $A_\n$ in the following short exact sequence of abelian groups:
\[
0\to\mathbb{Z}^{r-1}\overset{A_\n}{\longrightarrow}\mathbb{Z}^{r}\to\Gamma_{\mathbf{n}}^{\text{ab}}\to0
\]
whose representation in a canonical basis is given by the $r\times(r-1)$ matrix: 
\[
A_\n=\left(\begin{array}{cccc}
-n_{1} & 0 & \cdots & 0\\
n_{2} & -n_{2} & \ddots & \vdots\\
0 & n_{3} & \ddots & 0\\
\vdots & \ddots & \ddots & -n_{r-1}\\
0 & \cdots & 0 & n_{r}
\end{array}\right).
\]

\begin{lem}
\label{abel-lemma}
Let $\Gamma_{\n}$ be a generalized torus knot group.  Then $\Gamma_{\mathbf{n}}^{\text{ab}}\cong \Z$.  Consequently, $\hom(\Gamma_\n^{ab},G)\cong G$.
\end{lem}

\begin{proof}
A short proof is to compute the Smith normal form of the matrix $A_\n$ above (see also Proposition \ref{prop:abel-iff} below). However, we can also argue inductively on $r$ which gives a constructive proof as noted in Remark \ref{rem:const}. Indeed, $r=1$ is trivial.  Let us see how $r=2$ plays out.  Since $\gcd(n_1,n_2)=1$, Bezout's identity implies there are integers $s,t$ with $sn_1+tn_2=1$ and thus, using $n_1\gamma_1=n_2\gamma_2$ we get $$sn_2\gamma_2=sn_1\gamma_1=\gamma_1(1-tn_2)=\gamma_1 - tn_2\gamma_1$$
Since $\Gamma_{\mathbf{n}}^{\text{ab}}$ is abelian, this implies $\gamma_1=n_2(s\gamma_2+t\gamma_1)$.  In the same way, we get $\gamma_2=n_1(t\gamma_1+s\gamma_2)$.  Thus, $\Gamma_{\mathbf{n}}^{\text{ab}}$ is generated by the single element $s\gamma_2+t\gamma_1$. To show that $\Gamma_{\mathbf{n}}^{\text{ab}}=\Z$ we argue as follows. If $N(t\gamma_1+s\gamma_2)=0$ for some $N>0$, then $tN\gamma_1=-sN\gamma_2$ which implies $$tNn_2\gamma_1=-n_2sN\gamma_2=-n_1sN\gamma_1$$ and so $0=N(tn_2+sn_1)\gamma_1=N\gamma_1$, which is a contradiction since $\gamma_1$ does not have finite order.

Observe that with $x:=t\gamma_1+s\gamma_2$ we have $n_1\gamma_1=n_1n_2x=n_2\gamma_2$. Now assume there is such an $x$ that generates $\Gamma_{\mathbf{n'}}^{\text{ab}}$ with infinite order for $\mathbf{n'}=(n_1,...,n_{r'})$ with $1\leq r'<r$.  Thus, by considering the subgroup generated by $\{\gamma_1,...,\gamma_{r-1}\}$, there exists $x\in \Gamma_{\mathbf{n}}^{\text{ab}}$ so that some multiple of $x$ equals $\gamma_i$ and $n_1\cdots n_{r-1}x=n_i\gamma_i=n_r\gamma_r$ for all $1\leq i\leq r-1$.  Thus, $\Gamma_{\mathbf{n}}^{\text{ab}}=\langle x,\gamma_r\rangle$.  Then $\gcd(n_1\cdots n_{r-1},n_r)=1$ which implies by the $r=2$ case there exist $s,t\in\Z$ so that $\Gamma_{\mathbf{n}}^{\text{ab}}=\langle tx+s\gamma_r\rangle\cong \Z$.
\end{proof}

\begin{rem}\label{rem:const}
The above proof suggests a general formula for the generator $x$ of the group $\Gamma_{\mathbf{n}}^{\text{ab}}$, for general $r\geq 2$, in terms of the
$\gamma_{1},...,\gamma_{r}$. Indeed, one starts by observing that the pairwise coprime condition, $\gcd(n_{i},n_{j})=1$ for all $i\neq j$, is equivalent to the condition:
\[
\mathrm{gcd}(N/n_1,\cdots,N/n_r)=1,
\] 
where $N=n_1 n_2\cdots n_r$ $($see the following Lemma $\ref{lem:genbez})$. In this case, Bezout's identity gives integers $b_{1},...,b_{r}$ so that $b_{1}\frac{N}{n_{1}}+\cdots+b_{r}\frac{N}{n_{r}}=1$. With these integers, define
$x:=b_{1}\gamma_{1}+\cdots+b_{r}\gamma_{r}.$
Indeed, using $n_i\gamma_i = n_r\gamma_r$ for all $i$, we see:
\begin{eqnarray*}
n_{1}n_{2}\cdots n_{r-1}x & = & n_{1}n_{2}\cdots n_{r-1}b_{1}\gamma_{1}+\cdots+n_{1}n_{2}\cdots n_{r-1}b_{r}\gamma_{r}=\\
 & = & n_{2}\cdots n_{r-1}n_{r}b_{r}\gamma_{r}+\cdots+n_{1}n_{2}\cdots n_{r-1}b_{r}\gamma_{r}=\\
 & = & (N/n_{1}\,b_{1}+\cdots+N/n_{r}\,b_{r})\gamma_{r}=\gamma_{r}.
\end{eqnarray*}
Similarly, we obtain $\gamma_{j}=\frac{N}{n_{j}}x$ for all $j=1,\ldots,r$.
\end{rem}

\begin{lem}\label{lem:genbez}
Let $N=n_{1}\cdots n_{r}$. Then $\gcd(n_{i},n_{j})=1$ for all $i\neq j$
if and only if 
\[
\mathrm{gcd}(N/n_1,\ldots,N/n_r)=1.
\]
\end{lem}

\begin{proof}
Assume $\gcd(n_{i},n_{j})=1$ for all $i\neq j$. If $P_{i}$ is the set
of primes appearing in the factorization of $n_{i}$, then $P_{i}\cap P_{j}=\varnothing$
for all $i\neq j$. Letting $P=\bigcup_{i=1}^{r}P_{i}$ be the set
of primes in the factorization of $N$, the previous condition means
that also $\bigcap_{i=1}^{r}(P\setminus P_{i})=P\setminus P=\varnothing$.
Since $P\setminus P_{i}$ are the primes in $N/n_{i}$, we get that
the unique common divisor of all $N/n_{i}$ is 1, as claimed. Conversely,
if without loss of generality $d:=\gcd(n_{1},n_{2})>1$, then $d|n_{1}$ and $d|n_{2}$. Hence, 
$d\big|\frac{N}{n_{k}}$ for all $k$ since either $n_1\big|\frac{N}{n_{k}}$ or $n_2\big|\frac{N}{n_{k}}$.
Hence $d>1$ is a common divisor of all $N/n_{i}$ as wanted. 
\end{proof}

We now prove, for generalized torus link groups that are not generalized knot groups, Lemma \ref{abel-lemma} fails.

\begin{prop}\label{prop:abel-iff}
Let $\Gamma_\n$ be a generalized torus link group. Then, the abelianization $\Gamma_{\mathbf{n}}^{\text{ab}}$ is a free group
of rank $1$ if and only if $\Gamma_{\mathbf{n}}$ is a generalized torus knot group.
\end{prop}

\begin{proof}

Lemma \ref{abel-lemma} shows the backwards direction.  So we need only show that if there exists $i\not=j$ such that $\gcd(n_i,n_j)\not=1$, then $\Gamma_\n^{ab}\not\cong \Z$.

The rank of $A_\n$ is $r-1$, and so we have:
\[
\Gamma_{\mathbf{n}}^{\text{ab}}=\mathsf{Coker}A_\n\cong\mathbb{Z}\oplus F,\]
where $F$ is a finite abelian group of the form:
$F \cong \bigoplus_{j=1}^{r-1}\mathbb{Z}_{a_{i}} $ 
where $a_{1}\mid\cdots\mid a_{r-1}$ are the integers that appear in the Smith canonical factorization:
\[
B:=PA_\n Q=\left(\begin{array}{cccc}
a_{1} & 0 & \cdots & 0\\
0 & a_{2} & & \vdots\\
\vdots & \ddots & \ddots & 0\\
0 & & 0 & a_{r-1}\\
0 & \cdots & 0 & 0
\end{array}\right),
\]
where $P\in \GL(r,\mathbb{Z})$ and $Q\in \GL(r-1,\mathbb{Z})$ are invertible integer matrices.

Suppose, without loss of generality, that $n_{1}$ and $n_{2}$ are not coprime. Since the first step to realize the Smith normal form places in the $(1,1)$-position the greatest common denominator of the elements in the first row and column, then $a_{1}=\gcd(n_{1},n_{2})>1$.  Thus, $F$ is non-trivial, and hence $\Gamma_{\mathbf{n}}^{\text{ab}}$ is not free.
\end{proof}

When $r=2$, we have the following precise consequence.

\begin{cor}
$\Gamma_{(n,m)}^{ab}\cong \Z\oplus \Z_{\gcd(n,m)}$.
\end{cor}

Recall $\hom^*(\Gamma,G)\subset \hom(\Gamma, G)$ is the subspace of polystable homomorphisms, and the $G$-character variety of $\Gamma$ is the conjugation quotient $\X_\Gamma(G)=\hom^*(\Gamma,G)/G$.

The epimorphism $\Gamma_\n\to \Gamma_\n^{ab}$ induces inclusions $\hom(\Gamma_\n^{ab}, G)\hookrightarrow\hom(\Gamma_\n,G)$ and $\X_{\Gamma_\n^{ab}}(G)\hookrightarrow\X_{\Gamma_\n}(G)$ for any reductive $\C$-group $G$.  The image of this inclusion is called the {\it abelian locus} of the character variety and denoted $\X_{\Gamma_\n}^{ab}(G)$.  It consists of representations whose evaluations at generators pair-wise commute with each other.

\begin{thm}\label{abelian-thm}
Let $G$ be a connected, reductive $\C$-group, $DG$ its derived subgroup, and $ZG$ its center, and $F=DG\cap ZG$. Let $\Gamma_\n$ be a generalized torus knot group.  Then the abelian locus of $\X_{\Gamma_\n}(G)$ is isomorphic to $\X_{\Z}(G)\cong G\quot G\cong\C^{\mathrm{Rank}(DG)}/H\times_{F} (\C^*)^{\dim Z}$ where $H\leq Z(\widetilde{DG})$.  Consequently, the abelian locus is irreducible and so connected. 
\end{thm}

\begin{proof}
Lemma \ref{abel-lemma} implies directly that $\hom(\Gamma_\n^{ab},G)\cong G$ and $\X_{\Gamma_\n}^{ab}(G)\cong \X_{\Z}(G)\cong G\quot G$ for any reductive $\C$-group $G$.  Now if $G$ is connected, then the central isogeny theorem says $G\cong DG\times_F ZG$ where $F=DG\cap ZG$ is a finite central subgroup and $DG\times_F ZG$ is the quotient by $F$ acting on $DG\times ZG$ by $f\cdot (d,z)=(df^{-1},fz)$.  Thus, the GIT quotient of $G$ by $G$ acting by conjugation is $G\quot G=DG\quot DG\times_F ZG$.  By a result of Steinberg \cite{Steinberg}, $DG\quot DG$ is $\C^{\mathrm{Rank}(DG)}$ if $DG$ is simply-connected and it is $\C^{\mathrm{Rank}(DG)}/H$ for a finite group $H$ otherwise.  Concretely, $H$ is a subgroup of the center of the universal cover of $DG$.  Since $ZG$ is a central torus, it is a product of $\C^*$'s.  Lastly, since $G$ is connected it is irreducible and so $G\quot G$ is irreducible (and so connected) too.
\end{proof}

Let $\X_{\Gamma_\n}^{red}(G)$ denote the reducible locus in $\X_{\Gamma_\n}(G)$.

\begin{cor}
If $G$ is a non-abelian connected reductive $\C$-group that admits a faithful representation into $\GL(2,\C)$ or a quotient thereof, and $\Gamma_\n$ is a generalized torus knot group, then the reducible locus in $\X_{\Gamma_\n}(G)$ is irreducible $($and so connected$)$, and intersects every component of the irreducible locus of  $\X_{\Gamma_\n}(G)$.  In particular, $\X_{\Gamma_\n}^{red}(\GL(2,\C))\cong \C\times_{\Z_2} \C^*$,  $\X_{\Gamma_\n}^{red}(\SL(2,\C))\cong \C$, and  $\X_{\Gamma_\n}^{red}(\mathrm{PGL}(2,\C))\cong \C/\Z_2.$
\end{cor}

\begin{proof}
In this case, the reducible locus and the abelian locus coincide.  Thus, the result follows from Theorem \ref{abelian-thm} and the observation that every irreducible representation deforms to a reducible one, which is proved independently of this corollary with Theorem \ref{thm:path-conn}.
\end{proof}

\subsection{A useful presentation for generalized torus link groups.}
Recall from the previous section that \[
\Gamma_{\mathbf{n}}=\left\langle \gamma_{1},\ldots,\gamma_{r}\ |\  \gamma_{1}^{n_{1}}=\gamma_{2}^{n_{2}}=\cdots=\gamma_{r}^{n_{r}}\right\rangle.
\]
Given $x,y$ in a group, denote their commutator by $[x,y]:=xyx^{-1}y^{-1}$.
In the same terms as $\Gamma_{\mathbf{n}}$, we define 
\[
\widehat{\Gamma}_{\mathbf{n}}:=\left\langle \gamma_{1},\ldots,\gamma_{r},t\:|\:\gamma_{i}^{n_{i}}=t,\ [\gamma_{i},t]=1,\ 1\leq i\leq r\right\rangle ,
\]
which we now show is just a different presentation of $\Gamma_{\n}$.

\begin{lem}
Let $\n\in\N^{r}$. Then there is a group isomorphism $\psi:\widehat{\Gamma}_{\n}\to\Gamma_{\n}$. 
\end{lem}

\begin{proof}
 
To show the claim, we define $\psi(\gamma_{i})=\gamma_{i}$, for $1\leq i\leq r$
and $\psi(t)=\gamma_{1}^{n_{1}}$.  Note that $\psi$ indeed extends to
a group homomorphism since it preserves relations in the domain and codomain.  It is bijective since we can construct its inverse: $\psi^{-1}(\gamma_{i})=\gamma_{i}$.
\end{proof}

As a consequence, for any reductive $\C$-group $G$, we have isomorphisms $\hom(\Gamma_{\n},G)\cong\hom(\widehat{\Gamma}_{\n},G)$,
and $\X_{\Gamma_{\n}}(G)\cong\X_{\widehat{\Gamma}_{\n}}(G)$.

\subsection{An Exact Sequence with Torus Link Groups}

Let $\Z_{n_i}$ be the cyclic group of order $n_i$ and let $*\Z_{\n}:=*_{i=1}^r\Z_{n_i}$ be the free product of cyclic groups $\Z_{n_i}$, where $\n=(n_1,...,n_r)$.

\begin{lem}\label{exact-claim}
There is an exact sequence $0\to \Z \to \widehat{\Gamma}_\n\to *\Z_\n\to 0$.
\end{lem}

\begin{proof}
  Since $t$ is central, it defines a normal subgroup isomorphic to $\Z$.  So we have an exact sequence $0\to \Z \to \widehat{\Gamma}_\n\to \widehat{\Gamma}_\n/\Z\to 0$.  However, setting $t$ to the identity demands that each generator $\overline{\gamma_i}\in \widehat{\Gamma}_\n/\Z$ has order $n_i$.  As there are clearly no other relations, we see $\widehat{\Gamma}_\n/\Z\cong *\Z_\n.$  That completes the proof.
\end{proof}

Before the next claim, we prove the following elementary lemma.

\begin{lem}\label{exact-lemma}
 Let $A,B,C,G$ be groups and suppose $1\to A\to B\to C\to 1$ is exact.  Then $\id\to \hom(C,G)\to \hom(B,G)\to \hom(A,G)$ is ``left exact".  $G$ acts on the hom-sets by conjugation, and the induced maps on the hom-sets are equivariant with respect to this action.  Moreover, if $A,B,C,G$ are topological groups and the exact sequence is given by continuous maps, then the induced maps on the hom-sets are continuous with respect to the compact-open topology on the hom-sets.
\end{lem}

\begin{proof}
Since we are not in an abelian category, this proof will spell out what we mean by ``left exact" in this context.  

First, let $\alpha$ be the homomorphism from $A\to B$, and $\beta$ the homomorphism from $B\to C$.  Then the maps $\beta^*:\hom(C,G)\to \hom(B,G)$ and $\alpha^*:\hom(B,G)\to \hom(A,G)$ are given by $f\mapsto f\circ \beta$ and $f\mapsto f\circ \alpha$ respectively.  

Now suppose $f\circ \beta=g\circ \beta$. Then, since $\beta$ is surjective it has a right inverse and so  $f=g$ and $\beta^*$ is injective.  We define $\ker(\alpha^*)$ to be the set of all homomorphism that are mapped to the trivial homomorphism by $\alpha^*$.  Since $\beta\circ \alpha$ is the trivial homomorphism, we see that $\im(\beta^*)\subset \ker(\alpha^*)$.

Next, suppose $\alpha^*(g)$ is the trivial homomorphism, that is, $g(\alpha(a)))=1$ for all $a\in A$.  Then we need to show that $g=h\circ \beta$ for some homomorphism $h:C\to G$.  Since $\beta$ is onto, for every $c\in C$ there is $b_c\in B$ so that $\beta(b_c)=c$; so define $h(c)=g(b_c)$.  By definition, $h\circ \beta =g$, so we need to show that $h$ is well-defined and a homomorphism. Suppose $\beta(b_c)=\beta(b_c')$, then $\beta(b_c^{-1}b_c')=1$ which implies by exactness that $b_c^{-1}b_c'=\alpha(a)$ for some $a\in A$.  Thus, $g(b_c^{-1}b_c')=g(\alpha(a))=1$ since $g\circ \alpha$ is the trivial homomorphism.  Thus, $g(b_c)=g(b_c')$ and we see $h$ is a well-defined set function.  Likewise, for $c_1,c_2\in C$ we have $h(c_1c_2)=g(b_{c_1c_2})$, $h(c_1)=g(b_{c_1}),$ and $h(c_2)=g(b_{c_2})$ where $\beta(b_{c_1c_2})=c_1c_2=\beta(b_{c_1})\beta(b_{c_2})=\beta(b_{c_1}b_{c_2}).$  And so, since $h$ is well-defined $h(c_1c_2)=g(b_{c_1c_2})=g(b_{c_1}b_{c_2})=g(b_{c_1})g(b_{c_2})=h(c_1)h(c_2),$ as required.

Notice that that $\alpha^*$ and $\beta^*$ are $G$-equivariant:  $$\alpha^*(g\cdot f)(a)=gf(\alpha(a))g^{-1}=g\cdot \alpha^*(f)(a),$$ and likewise for $\beta^*$.  Thus $G$-orbits map to $G$-orbits. 

Lastly, given a sub-basic set $U(K,O)=\{f\in \hom(A,G)\ |\ f(K)\subset O\}$ where $K$ is compact in $A$ and $O$ is open in $G$, we have that $(\alpha^*)^{-1}(U(K,O))=U(\alpha(K),O)\subset \hom(B,G)$ since $\alpha$ is continuous implies $\alpha(K)$ is compact in $B$.  Hence, $\alpha^*$ is continuous (likewise for $\beta^*$).
 
\end{proof}

Let $G_{ss}$ be the collection of semisimple elements in $G$, that is, the polystable locus of $\hom(\Z,G)$.

\begin{thm}\label{flawed-thm}Let $G$ be a reductive $\C$-group. There is an exact sequence $$\id\to\hom^*(*\Z_\n, G)\hookrightarrow \hom^*(\widehat{\Gamma}_\n,G)\twoheadrightarrow \hom^*(\Z,G)\cong G_{ss}\to \id.$$
\end{thm}

\begin{proof}
Recall $\widehat{\Gamma}_{\mathbf{n}}=\left\langle \gamma_{1},\ldots,\gamma_{r},t\:|\:\gamma_{i}^{n_{i}}=t,\ [\gamma_{i},t]=1,\ 1\leq i\leq r\right\rangle.$ 

The exact sequence from Lemma \ref{exact-claim} gives a left exact sequence of hom-spaces by Lemma \ref{exact-lemma}:
 $$\id\to\hom(*\Z_\n, G)\hookrightarrow \hom(\widehat{\Gamma}_\n,G)\to \hom(\Z,G)\cong G.$$ 
 
Since $\hom(*\Z_\n, G)\hookrightarrow \hom(\widehat{\Gamma}_\n,G)$ is a continuous $G$-equivariant inclusion, we see that closed $G$-orbits in $\hom(*\Z_\n, G)$ map to closed $G$-orbits in $\hom(\widehat{\Gamma}_\n,G)$. 

So, with respect to the restrictions of the induced maps, we have $\id\to \hom^*(*\Z_\n, G)\hookrightarrow \hom^*(\widehat{\Gamma}_\n,G)\to  \hom(\Z,G)$ is left exact (note middle exactness follows again since $$\hom(*\Z_\n, G)\hookrightarrow \hom(\widehat{\Gamma}_\n,G)$$ is a continuous $G$-equivariant inclusion).

Every element $g\in G_{ss}$ is in a maximal torus.  Thus, it is closed under $n$-th roots.  Let $g_i$ be a $n_i$-th root of $g$ (necessarily semisimple).  Then this defines a homomorphism $h\in \hom(\widehat{\Gamma}_{\mathbf{n}}, G)$ by setting $h(\gamma_i):=g_i$ for all $i$ that maps to $g$ by restriction to $t\in\widehat{\Gamma}_{\mathbf{n}}$. Since the components of $h$ are all semisimple and commute (as they are in the same maximal torus), the conjugation orbit of $h$ is closed. 

On the other hand, any homomorphism in $\hom^*(\widehat{\Gamma}_\n,\GL(m,\C))$ has its value at $t$ semisimple by Proposition \ref{prop:solutions}.  Therefore, any homomorphism in $\hom^*(\widehat{\Gamma}_\n,G)$ also has its value at $t$ semisimple since $G$ admits a faithful linear representation into $\GL(m,\C)$ for some $m$, and semisimplicity of an element of $G$ is detected by any faithful linear representations of $G$ by \cite{borel}.

Thus, after restricting to polystable points the left-exact sequence becomes fully exact:
$$\id\to\hom^*(*\Z_\n, G)\hookrightarrow \hom^*(\widehat{\Gamma}_\n,G)\twoheadrightarrow \hom^*(\Z,G)\cong G_{ss}\to \id.$$
 \end{proof}

\begin{rem}\label{rem:flawed}

Let $\Gamma$ be a finitely generated group and let $G$ be a reductive $\C$-group.  Let $K$ be a maximal compact subgroup of $G$.  Then we say $\Gamma$ is {\it $G$-flawed} if there exists a strong deformation retraction from the character variety $\X_\Gamma(G)$ to $\X_\Gamma(K)$.  We say $\Gamma$ is {\it flawed} if $\Gamma$ is $G$-flawed for all $G$.

In \cite{FLflawed}, it was observed that torus knot groups are $\SL(2,\C)$-flawed from results in \cite{munozsl2, munozsu2}.  Motivated by this it was conjectured in \cite{FLflawed} that {\it generalized} torus link groups are flawed.  In \cite{sl3torus} some further evidence for this conjecture is given by showing {\it some} torus knots groups are $\SL(3,\C)$-flawed.

Theorem \ref{flawed-thm} says, loosely speaking, that $G$-character varieties of generalized torus link groups are built from $G$-character varieties of free products of finitely many cyclic groups.  Since any group $\Gamma$ that is isomorphic to a free product of finitely many cyclic groups is flawed by \cite{FLflawed}, one might
expect that generalized torus link groups are flawed too.  One possible strategy to establish this using Theorem \ref{flawed-thm} would be to argue that the exact sequence is a $($stratified$)$ fibration and then argue that since the base and fibers deformation retract as required by \cite{FLflawed}, the total spaces do as well.
\end{rem}

\subsection{Path-connectedness of $\mathfrak{X}_{\Gamma_{\mathbf{n}}}(G)$}

In this subsection, we prove that both the representation varieties
$\hom(\Gamma_{\mathbf{n}},G)$ and the character varieties $\mathfrak{X}_{\Gamma_{\mathbf{n}}}(G)$
are path-connected (generally, with many irreducible components),
for every generalized torus {\it knot} group $\Gamma_{\mathbf{n}}$ and
reductive $\mathbb{C}$-group $G$. 

We start with some generic observations. Fix $r\in\mathbb{N}$, and
let $F_{r}$ denote the free (non-abelian) group of rank $r$. Write
$\mathbf{x}=(x_{1},...,x_{r})\in G^{r}$, and denote by $\langle\mathbf{x}\rangle=\langle x_{1},\ldots,x_{r}\rangle$
the Zariski closure of the subgroup of $G$ generated by the elements $x_{1},\ldots,x_{r}\in G$.

\begin{prop}
\label{prop:Rich1} The subgroup $H:=\langle\mathbf{x}\rangle\subset G$
is a reductive $\mathbb{C}$-group if and only if $\mathbf{x}=(x_{1},\ldots,x_{r})$
is a polystable element in $\hom(F_{r},G)\cong G^{r}$.
\end{prop}

\begin{proof}
This is a main result in Richardson \cite[Thm 3.6]{Rich}. Observe that the tuple
$\mathbf{x}$ is there defined to be \emph{semisimple} if $H=\langle\mathbf{x}\rangle$
is a reductive $\mathbb{C}$-group.
\end{proof}

For a finitely generated group $\Gamma$, the choice of generators
$\gamma_{1},\ldots,\gamma_{r}$ of $\Gamma$ produces an inclusion
of varieties
\[
\hom(\Gamma,G)\subset\hom(F_{r},G),
\]
given by evaluation on the generators.  Precisely, the image of a representation
$\rho:\Gamma\to G$, is identified with the $r$-tuple $\mathbf{x}=(x_{1},...,x_{r})\in G^{r}\cong\hom(F_{r},G)$
with $x_{i}=\rho(\gamma_{i})$.

\begin{lem}
\label{lem:closed} Let $\rho\in\hom(\Gamma,G)$ be polystable. Then
$\mathbf{x}=(x_{1},...,x_{r})=(\rho(\gamma_{1}),...,\rho(\gamma_{r}))$
is polystable in $\hom(F_{r},G)\cong G^{r}$.
\end{lem}

\begin{proof}
This is clear, since $\hom(\Gamma,G)\subset\hom(F_{r},G)$ is a closed
topological embedding, so that closedness of orbits $G\cdot\mathbf{x}\cong G\cdot\rho$
is preserved.
\end{proof}

We now consider generalized torus link groups $\Gamma_\n$. The proof of the
following lemma is deferred to Section \ref{sec:la} (Lemma \ref{lem:semisimple-again}). 

\begin{lem}
\label{lem:semisimple} Let $\Gamma_{\mathbf{n}}=\left\langle \gamma_{1},...,\gamma_{r}\,|\,\gamma_{i}^{n_{i}}=\gamma_{j}^{n_{j}}\right\rangle $
be a generalized torus \emph{link} group and let $\rho\in\hom(\Gamma,G)$
be polystable. Then all $x_{i}:=\rho(\gamma_{i})\in G$ are semisimple
elements.
\end{lem}

\begin{thm}
\label{thm:path-conn}
Let $\Gamma_{\mathbf{n}}=\left\langle \gamma_{1},...,\gamma_{r}\,|\,\gamma_{i}^{n_{i}}=\gamma_{j}^{n_{j}}\right\rangle $
be a generalized torus \emph{knot} group. The space $\hom^{*}(\Gamma_\n,G)$
is path-connected. 
\end{thm}

\begin{proof}
We recall that $G$ is a reductive $\C$-group and that $ZH$ denotes the center of a subgroup $H\subset G$.

Let $\rho\in\hom^*(\Gamma_\n,G)$ be polystable, and let $g=x_{1}^{n_{1}}=\cdots=x_{r}^{n_{r}}$.
Then, all $x_{i}$ are semisimple by Lemma \ref{lem:semisimple}, and so is
$g$. Now, $g$ commutes with all $x_{i}$ since $gx_i=x_i^{n_i+1}=x_ig$, and hence $g$ commutes with the group generated by $\{x_1,...,x_r\}$. Thus, by \cite[Section 3]{Rich}, $g$ commutes with the Zariski closure of the group generated by $\{x_1,...,x_r\}$ which we denote by $H:=\langle\mathbf{x}\rangle$. We conclude $g\in H\cap Z_G(H)=ZH.$

By Lemma \ref{lem:closed}
and Proposition \ref{prop:Rich1}, $H$ is a reductive $\mathbb{C}$-group,
so it has a maximal torus $T_{H}$. In fact, $T_{H}$
is a maximal torus of $H_{0}$, the identity component of $H$. Since each $x_i$ is semisimple, each $x_i$ is contained in a maximal torus.  But since all maximal tori are conjugate, each $x_i$ is conjugate to an element in $T_H$.

Now, since $H_{0}$ is connected, for every $i$ there exists a path
$h_{i}(t)\in H_{0}$ such that $h_{i}(0)=e$ and $h_{i}(1)$ satisfies:
\[
h_{i}(1)\,x_{i}\,h_{i}(1)^{-1}\in T_{H}
\]
(since all tori are conjugate in $H_{0}$, by elements of $H_{0}$).
Then, with $x_{i}(t):=h_{i}(t)\,x_{i}\,h_{i}(t)^{-1}$ we have 
\[
x_{i}(t)^{n_{i}}=h_{i}(t)\,g\,h_{i}(t)^{-1}=g=h_{j}(t)\,x_{j}^{n_{j}}\,h_{j}(t)^{-1}=x_{j}(t)^{n_{j}}
\]
 (since $g\in ZH$, and $x_{i}\in H$) and we have
\[
\mathbf{x}(t)=(x_{1}(t),...,x_{r}(t))\in\hom(\Gamma_\n,G),\text{ with }\mathbf{x}(1)\in\hom(\Gamma_\n,T_{H}).
\]
The result now follows, since 
\[
\hom(\Gamma_\n,T_{H})\subset\hom(\Gamma_\n,T_{G})\subset\hom(\Gamma_\n^{ab},G),
\]
and by Theorem \ref{abelian-thm} the abelian locus $\hom(\Gamma_\n^{ab},G)$
is irreducible, hence path-connected.
\end{proof}

Since $G$ is path-connected and $\mathfrak{X}_{\Gamma_{\mathbf{n}}}(G)$
can be identified with the polystable quotient $\mathfrak{X}_{\Gamma_{\mathbf{n}}}(G)\cong\hom^{*}(\Gamma_\n,G)/G$,
the following is clear.

\begin{cor}
For any generalized torus \emph{knot} group $\Gamma_{\mathbf{n}}$,
$\mathfrak{X}_{\Gamma_{\mathbf{n}}}(G)$ is path-connected.
\end{cor}

\section{Generalized Torus Link Groups and Linear Algebra}\label{sec:la}

In the case of $G=\GL(m,\mathbb{C})$ it is easy to see that $\rho$
is polystable if and only if it can be written as a direct sum:
${\displaystyle \rho=\bigoplus_{j=1}^{s}\rho_{j},}$
where all the $\rho_{j}\in\hom(\Gamma,\GL(m_{j},\mathbb{C}))$ are
irreducible, for a given partition of $m$ as $m=m_{1}+\cdots+m_{s}$.

This defines a stratification of $\hom^*(\Gamma, G)$.  For the top stratum, let $\hom^{irr}(\Gamma, G)$ be the subspace of irreducible representations; necessarily a Zariski open subset of the affine variety $\hom(\Gamma, G)$.  Given a partition $P=(m_1,...,m_s)$, we define  $$\hom^P(\Gamma, G):=\{\rho\in \hom(\Gamma,G)\ |\ \exists g\in G\text{ so }g\rho g^{-1}\in \bigoplus_{j=1}^s\hom^{irr}(\Gamma,\GL(m_{j},\mathbb{C}))\}.$$  For example, $\hom^{irr}(\Gamma, G)=\hom^{(m)}(\Gamma, G)$.  Clearly these strata are disjoint from each other, and since polystability and complete-reducibility coincide, we conclude $\hom^*(\Gamma, G)=\sqcup_{P}\hom^P(\Gamma, G)$.  Since reducibility of fixed parabolic type is a Zariski closed condition and irreducibility is a Zariski open condition, each of these strata is locally closed in the Zariski topology.

We need the following elementary lemma.

\begin{lem}\label{lem:eig}
Let $A\in \GL(m,\mathbb{C})$, and $Eig(A)$ denote
the span of the eigenvectors of $A$. Then, for all $k$ in $\mathbb{Z}- \{0\}$, $Eig(A)=Eig(A^{k})$.
\end{lem}

\begin{proof}
It is easy to see that $Eig(A)\subset Eig(A^{k})$ since for every eigenvector $v\in Eig(A)$ with eigenvalue $\lambda$ we have $A^kv=\lambda^k v$.

For the converse, we can assume that $A$ is in Jordan form since any change of basis for $A$ given by $g\in\GL(m,\C)$ determines a like change of basis for $A^k$, i.e. $(gAg^{-1})^k=gA^kg^{-1}$.  So $A=J_{1}\oplus\cdots\oplus J_{n}$
where each $J_{i}$ is a Jordan block with eigenvalue $\lambda_{i}$. Then, for each block we have only a one dimensional eigenspace with eigenvalue $\lambda_{i}$, and it is clear that $A^{k}=J_{1}^{k}\oplus\cdots\oplus J_{n}^{k}$.  So the result follows if it holds for $n=1$.  

So now assume $A=J_1$ and thus $A=\lambda I+N$ where $I$ is the identity, $N$ is a nilpotent matrix, and $\lambda=\lambda_1$ is the eigenvalue of $A$.  Then $A^k=\lambda^k I+\sum_{i=1}^k\lambda^{k-i}c_i N^i$ for appropriate binomial coefficients $c_i$.  Let $w$ be an eigenvector of $A^k$ with eigenvalue $\mu$, and denote $M=\sum_{i=1}^k\lambda^{k-i}c_i N^i$.  Then $\mu w=A^kw=\lambda^k w+ M w$ which implies that $(\mu-\lambda^k)w=M w$ and so $w$ is an eigenvector of $M$ with eigenvalue $\mu-\lambda^k$.  But $M$ is nilpotent and thus $\mu=\lambda^k$ and $w$ is in the kernel of $M$.  Obviously, $\mathrm{Ker}(N)\subset \mathrm{Ker}(M).$  On the other hand, $N$'s upper-diagonal consists of 1's and $N$ has 0's everywhere else.  So powers of $N$ shift the upper-diagonal up.  Therefore, $M$ and $N$ have the same pivot positions and hence are row equivalent.  Thus,  $\mathrm{Ker}(N)=\mathrm{Ker}(M).$  But if $w\in \mathrm{Ker}(N)$, then $w\in Eig(A)$ with eigenvalue $\lambda$.
\end{proof}

We denote by $\Gamma_\n$ a generalized torus link group, that is, a group admitting a presentation \[
\Gamma_{\mathbf{n}}:=\left\langle \gamma_{1},\ldots,\gamma_{r}\:|\:\gamma_{1}^{n_{1}}=\gamma_{2}^{n_{2}}=\cdots=\gamma_{r}^{n_{r}}\right\rangle,
\] where $n_i\geq 1$ for all $1\leq i\leq r$.

\begin{prop}\label{prop:diagonalization}
Let $G$ be either $\SL(m,\C)$ or $\GL(m,\C)$, and let $\rho\in\hom^*(\Gamma_{\mathbf{n}},G)$ with corresponding
partition $m=m_{1}+\cdots+m_{s}$. Then, every element $\rho(\gamma_{i})$
is diagonalizable and there exists a basis for $\C^m$ and scalars $w_{1},\ldots,w_{s}\in\mathbb{C}^{*}$
such that in this basis:
\[
\rho(\gamma_{i})^{n_{i}}=\bigoplus_{j=1}^{s}\omega_jI_{m_{j}},
\]
for all $i=1,\ldots,r$, where $I_{m}$ denotes the identity matrix
of size $m$.\end{prop}

\begin{proof}
It suffices to prove the proposition for  $\GL(m,\C)$.  As $\rho$ is assumed to be polystable, we can decompose $\rho$
as a direct sum of irreducibles; that is, each $\rho(\gamma_{i})$ for $1\leq i\leq r$ is a direct
sum of blocks of size $m_{j}$ for $1\leq j\leq s$. And as diagonalization is a
property preserved by direct sums, we are reduced to consider the irreducible case; that is, we assume $s=1$.  In this case, either $r>1$ or $m=1$.  The $m=1$ case is trivial so further assume $r\geq 2$.

Now write $A_{i}:=\rho(\gamma_{i})\in \GL(m,\mathbb{C})$ for all $i=1,\ldots,r.$ From the definition of $\Gamma_{\n}$ we have:
\begin{equation}\label{eigequal}
A_{1}^{n_{1}}=A_{2}^{n_{2}}=\cdots=A_{r}^{n_{r}}.
\end{equation}
Thus, according to Lemma \ref{lem:eig}, for all $i,j$:
$$
Eig(A_{i})=Eig(A_{i}^{n_{i}})=Eig(A_{j}^{n_{j}})=Eig(A_{j}).
$$
Let $W$ be this common vector space. Then $A_{i}(W)\subset W$ for all $i$, and since $\rho$ is irreducible, by Schur's Lemma, $W=\mathbb{C}^{m}$. This means that $A_{i}$ is diagonalizable for each $i$, and consequently, $A_i^{n_i}$ is diagonalizable for each $i$ too.  Since each $A_i^{n_i}$ commutes with each $A_j$ by Equations \ref{eigequal}, Schur's Lemma also implies there exists $\omega\in \C^*$ so $A_{1}^{n_{1}}=\cdots=A_{r}^{n_{r}}=\omega I_{m}$.
\end{proof}

Denote $[r]:=\{1,...,r\}.$

\begin{cor}\label{important-cor}
Let $\rho\in\hom^{*}(\Gamma_\n,\GL(m,\mathbb{C}))$ be an \emph{ irreducible}
representation. Then, every element $\rho(\gamma_i)$ is diagonalizable
and there exists $\omega\in\mathbb{C}^{*}$ such that: $\rho(\gamma_{i})^{n_{i}}=\omega I_m$, for all $i\in [r]$.
\end{cor}

\begin{rem}
The converse of the above corollary is not always true. For example, the trivial representation $\rho(\gamma)=I_m$ for all $\gamma\in \Gamma_\n$, is reducible and $\rho(\gamma_i)^{n_i}=\omega I_m$ for all $i$ with $\omega=1$.
\end{rem}

Given the above corollary, it is natural to ask what are the matrix solutions to $A^n=\omega I$.  According to \cite[page 173, Theorem 7.1]{Higham:2008:FM} the solutions are exactly $BDB^{-1}$ where $B\in\GL(m,\C)$ and $D\in \GL(m,\C)$ is diagonal such that $D^n=\omega I$. We now prove the comparable statement in our setting.

\begin{prop}\label{prop:solutions}Let $G$ be either $\GL(m,\C)$ or $\SL(m,\C)$.
Let $F_r$ be a rank $r$ free group, and let ${\displaystyle \rho=\bigoplus_{j=1}^{s}\rho_{j}}$ be in $\hom^*(F_r,G)$ with respect to a partition $m=m_{1}+\cdots+m_{s}$.  Then $\rho\in \hom^*(\Gamma_\n,G)$ if and only if for each $i,j$, $\rho_j(\gamma_i)=B_{i,j}D_{i,j}B_{i,j}^{-1}$ where $B_{i,j}\in \GL(m_j,\C)$, $D_{i,j}\in\GL(m_j,\C)$ or $\SL(m_j,\C)$ is diagonal, and for each $j$, $D_{i,j}^{n_i}=D_{k,j}^{n_k}=\omega_jI_{m_{j}}$ for all $i,k$ for some $\omega_j\in \C^*$. 
\end{prop}

\begin{proof}
It suffices to prove the result for $G=\GL(m,\C)$. First suppose $\rho\in\hom(\Gamma_{\n},\GL(m,\C))$ is polystable. Proposition \ref{prop:diagonalization} says every $\rho_j$ in  $\hom(\Gamma_\n,\GL(m_j,\C))$ has diagonalizable components $\rho_j(\gamma_i)$ and $\rho_j(\gamma_i)^{n_i}=\omega_jI_{m_j}$ for some $\omega_j\in \C^*$ for all $i$.  So there exists $B_{i,j}\in \GL(m_j,\C)$ and diagonal $D_{i,j}\in \GL(m_j,\C)$ so $A_{i,j}:=\rho_j(\gamma_i)=B_{i,j}D_{i,j}B_{i,j}^{-1}$ for each $i,j$.  Thus, for all $i,j$: $$B_{i,j}D_{i,j}^{n_i}B_{i,j}^{-1}=(B_{i,j}D_{i,j}B_{i,j}^{-1})^{n_i}=\omega_jI_{m_{j}},$$ and so $D_{i,j}^{n_i}=B_{i,j}^{-1}(\omega_jI_{m_{j}})B_{i,j}=\omega_jI_{m_{j}}$ for all $i,j$.  Therefore, for each $j$, $D_{i,j}^{n_i}=D_{k,j}^{n_k}$ for all $i,k$.  

Conversely, suppose for each $j$, $D_{i,j}^{n_i}=D_{k,j}^{n_k}$ for all $i,k$.  Then define ${\displaystyle \rho=\bigoplus_{j=1}^{s}\rho_{j}}$ where $\rho_j(\gamma_i):=B_{i,j}D_{i,j}B_{i,j}^{-1}$ for any choice of $B_{i,j}\in \GL(m_j,\C)$ that preserves polystability.  Proposition \ref{prop:diagonalization} implies for each $j$ there exists $\omega_j\in \C^*$ so that $D_{i,j}^{n_i}=\omega_jI_{m_{j}}$ for all $i$.  Running the calculations in the previous paragraph in reverse shows that ${\displaystyle \rho:=\bigoplus_{j=1}^{s}\rho_{j}}$ is an element of $\hom(\Gamma_\n,\GL(m,\C))$; it is polystable by construction.
\end{proof}

By Proposition \ref{prop:solutions}, for any partition $P$ of $m$, and  $\rho \in \hom^P(\widehat{\Gamma}_{\n},G)$, up to conjugation, $\rho(t)=\oplus_{j=1}^s\omega_jI_{m_j}$.  Denote by
\[
\hom^P_1(\widehat{\Gamma}_{\mathbf{n}},G)
\]
the representations in  $\hom^P(\widehat{\Gamma}_{\mathbf{n}},G)$ such that $|\omega_j|=1$ for all $j$.

\begin{lem}
\label{lem:def-S1}For any partition $P$, there is a $G$-equivariant SDR from $\hom^P(\hat{\Gamma}_{\mathbf{n}},G)$
to the space $\hom^P_1(\hat{\Gamma}_{\mathbf{n}},G)$.
\end{lem}

\begin{proof}
Since $\rho\in\hom^P(\hat{\Gamma}_{\mathbf{n}},G)$ is a direct sum of irreducible sub-representations and a direct sum of SDRs is an SDR, it suffices to show the irreducible case.

An element $\rho\in\hom^{irr}(\hat{\Gamma}_{\mathbf{n}},G)$ is defined
by its value on the generators $\gamma_{1},\ldots,\gamma_{r},t$.
Suppose these values form the tuple $$(A_{1},\ldots,A_{r},B)\in G^{r+1}.$$
These satisfy $A_{i}^{n_{i}}=B$, for every $1\leq i\leq r$. By Proposition
\ref{prop:diagonalization} and the irreducibility hypothesis, $B=\omega I_{m}$ with $\omega\in\mathbb{C}^{*}$.
Write $\lambda=|\omega|\in\mathbb{R}_{\geq0}$. Then, the homotopy
\begin{eqnarray*}
[0,1]\times\hom^{irr}(\hat{\Gamma}_{\mathbf{n}},G) & \to & \hom^{irr}(\hat{\Gamma}_{\mathbf{n}},G)\\
(s,A_{1},\ldots,A_{r},B) & \mapsto & (A_{1}\lambda^{-\frac{s}{n}},\ldots,A_{r}\lambda^{-\frac{s}{n}},\lambda^{-s}\omega I),
\end{eqnarray*}
is well-defined since $(A_{i}\lambda^{-\frac{s}{n}})^{n}=A_{i}^{n}\lambda^{-s}=\lambda^{-s}\omega I$.
At $s=0$ the map is the identity, and at $s=1$ it sends $B$ to
$\lambda^{-1}\omega I=\frac{\omega}{|\omega|}I$, so we are done.
Note finally that this homotopy is $G$-equivariant as it only involves
multiplication of each generator by central elements.
\end{proof}

After dealing with the $\GL(m_j,\C)$ case, we can now complete the proof of Theorem \ref{thm:path-conn} by proving Lemma \ref{lem:semisimple}.

\begin{lem}[Lemma \ref{lem:semisimple}]
\label{lem:semisimple-again} Let $\Gamma_{\mathbf{n}}=\left\langle \gamma_{1},...,\gamma_{r}\,|\,\gamma_{i}^{n_{i}}=\gamma_{j}^{n_{j}}\right\rangle $
be a generalized torus \emph{link} group and let $\rho\in\hom(\Gamma,G)$ be polystable. Then all $x_{i}:=\rho(\gamma_{i})\in G$ are semisimple elements.
\end{lem}

\begin{proof} Consider the abstract Jordan decomposition $x=su$ of an element $x\in G$, where $s$ is semisimple and $u$ is unipotent. It is well-known that \cite{borel}, for a faithful
representation $\phi:G\to \GL(n,\C)$, $\phi(x)=\phi(s)\phi(u)$ is the
Jordan decomposition of $\phi(x)\in \GL(n,\C)$.  We can take $\phi$ to be a closed immersion \cite{milne}. Then
$\rho$ being $G$-polystable is equivalent to the image $\rho(\Gamma)$ being completely reducible which is equivalent to the Zariski closure $\overline{\rho(\Gamma)}$ being reductive \cite{Rich}.  Since $\phi$ is both a closed
immersion and also faithful we know that $\phi\left(\overline{(\rho(\Gamma))}\right)$ is
also reductive and is the Zariski closure of $\phi(\rho(\Gamma))$.  Thus, $\phi(\rho(\Gamma))$ is completely reducible in $\GL(n,\C)$ and thus $\phi\circ \rho$ is $\GL(n,\C)$-polystable.  Let $\rho(\gamma_i)=x_{i}:=s_{i}u_{i}$ be the Jordan decomposition for each $i$.  Since $\rho$ is polystable we just showed that this implies that $\phi\circ\rho$ is polystable too, and thus Proposition \ref{prop:diagonalization} implies $y_{i}:=\phi(x_{i})$ is semisimple for all $i$. So $\phi(u_{i})$ is the identity matrix for all $i$. Since $\phi$ is a monomorphism we conclude that $u_{i}=e$, the identity element in $G$, for all $i$ too. Therefore, $x_{i}$ is semisimple for all $i$ as required.
\end{proof}

\section{Character Varieties of Generalized Torus Link Groups} \label{sec:irrcomp}

\subsection{Character variety of $*\mathbb{Z}_{\mathbf{n}}$}

Let us now consider $*\mathbb{Z}_{\mathbf{n}}=\mathbb{Z}_{n_{1}}*\mathbb{Z}_{n_{2}}*\cdots*\mathbb{Z}_{n_{r}}$,
the free product of cyclic groups of orders $n_{i}$. From Theorem \ref{flawed-thm}, the character varieties $\X_{\Gamma_{\mathbf{n}}}(G)$ and $\X_{*\mathbb{Z}_{\mathbf{n}}}(G)$ are closely related. The description of the irreducible components of the latter is easier, as we now explain.

Firstly, observe that $\hom(*\mathbb{Z}_{\mathbf{n}},G)$ is given
by a cartesian product:
\[
\hom(*\mathbb{Z}_{\mathbf{n}},G)=\times_{i=1}^{r}\hom(\mathbb{Z}_{n_{i}},G).
\]
Moreover, letting $I\in G$ be the identity, we have:
\[
\hom(\mathbb{Z}_{n},G)=\sqrt[n]{I}^G,
\]
as discussed in Appendix \ref{appendix:roots}.

\begin{prop} 
Let $G=\GL(m,\C)$. There are $N:=\prod_{i=1}^{r}\binom{m+n_{i}-1}{m}$
distinct irreducible components in $\hom(*\mathbb{Z}_{\mathbf{n}},\GL(m,\C))$.
Each component contains an abelian representation $\rho$ with $\rho(\gamma)$
diagonal for every $\gamma\in*\mathbb{Z}_{\mathbf{n}}$, and is isomorphic
to a homogeneous space of the form:
\[
G/H_{1}\times\cdots\times G/H_{r}
\]
where $H_{r}\subset\GL(m,\C)$ are Levi subgroups. The character variety
$\X_{*\mathbb{Z}_{\mathbf{n}}}(\GL(m,\C))$ has also $N$ irreducible
components, of dimensions between $0$ and $(m-1)(rm-m-1)$. 
\end{prop}

\begin{proof}
Let $T\subset \GL(m,\C)$ be the maximal torus consisting of diagonal matrices, and choose $\xi\in T$ with only $n$-th
roots of unity on the diagonal. From Appendix \ref{appendix:roots}, every irreducible component
of 
\[
\hom(\mathbb{Z}_{n},\GL(m,\C))=\sqrt[n]{I_{m}}
\]
is of the form
\[
G\,\xi\,G^{-1}.
\]
Since these sets are closed orbits, they are either disjoint or coincide.
The conjugation action identifies reordering of choices of $n$-th
roots, and hence the number of such choices is $\binom{m+n-1}{m}=\binom{m+n-1}{n-1}$.
So, for $\hom(*\mathbb{Z}_{\mathbf{n}},G)=\times_{i=1}^{r}\sqrt[n_{i}]{I_{m}}$,
the total number is $N$, as components behave multiplicatively under
cartesian products. 

The orbit-stabilizer theorem states that $G\,\xi\,G^{-1}$ is a homogeneous
space of the form $G/H$ where $H$ is the stabilizer of $\xi$. Moreover,
since $\xi\in G$ is semisimple, $H$ is a Levi subgroup whose dimension
is $\sum_{j=1}^\ell k_{j}^{2}$, where $k_{1},...,k_\ell$ are the multiplicities of
each eigenvalue of $\xi$, since $H\cong \GL(k_{1},\C)\times \GL(k_{2},\C)\times\cdots\times \GL(k_\ell,\C)$.
Hence, the components of maximal dimension of $\hom(*\mathbb{Z}_{\mathbf{n}},G)$
are isomorphic to 
\[
(G/T)^{r}
\]
with dimension $r(m^{2}-m)$. So, the character variety has dimension
$(m-1)(rm-m-1)$. There are $n_{1}\cdots n_{r}$ zero-dimensional
components, corresponding to the choices of scalar $\xi$, in each
factor. 
\end{proof}

\begin{example}
Let us detail the case $m=2$, for which $T\cong(\mathbb{C}^{*})^{2}$.
All irreducible components contain $\rho=(\xi_{1},\ldots,\xi_{r})$
with $\xi_{i}\in T$, and are isomorphic to:
\[
(\GL(2,\C)/T)^{k}
\]
where $k$ is the number of $\xi_{i}$ which are \emph{not} central
matrices. So, the representation space has dimension $2r$ and the
character variety is $2r-3$ dimensional, as long as $r\geq2$.  We have the same dimension for the case of $G=\SL(2,\C)$.
\end{example}

\subsection{Character variety of $\Gamma_{(n,m)}$}

We now give a new geometric proof that the irreducible components of the irreducible locus of a torus knot in the case $G=\SL(2,\C)$ are each isomorphic to $\C$.  The original result is in \cite{munozsl2} and the proof is algebraic.

But first we prove the following more general lemma.  

\begin{lem}\label{lem:doublecoset}
Let $G$ be a reductive $\C$-group and $T$ a maximal torus in $G$. Let $n,m$ be coprime integers. Let $(g,h)\in \hom^*(\Z_n*\Z_m, G)$ and assume both $g$ and $h$ are regular.  Let $C/\!\!/G$ be a Zariski irreducible component of $\X(\Z_n*\Z_m, G)$ which contains $[(g,h)]$. Then $C/\!\!/G$ is isomorphic to a double coset space $T\backslash G/T$.
\end{lem}

\begin{proof}
 We have that $\hom(\Z_n*\Z_m, G)\cong A \times B$ where $A$ is the set of $n$-th roots of the identity matrix in $G$ and $B$ is the set of $m$-th roots of the identity matrix in $G$.  Since $G$ is connected, Corollary 2 of Proposition 4 in \cite{Mc} implies that each of $A$ and $B$ consist of a finite number of conjugacy classes. 
 
We know from Lemma \ref{lem:semisimple} that each conjugacy class in $A$ or $B$ is the class of a semisimple element in $G$.  Thus, the conjugation orbits in $A$ or $B$ are Zariski irreducible.  Therefore, any irreducible component (in the Zariski topology) in $A\times B$ is the product of two conjugation orbits $G\cdot g\times G\cdot h$. By our assumption, the elements $g,h$ are not only semisimple, but {\it regular}.  

The map $\tau_z:G\to G$ given by $x\mapsto xzx^{-1}$ defines an isomorphism (orbit-stabilizer theorem) between $G/S_z$ and the conjugation orbit $G\cdot z=\{yzy^{-1}\ |\ y\in G\}$ where $S_z=\{y\in G\ |\ yzy^{-1}=z\}$ is the stabilizer of $z$.  We note that the action of $S_z$ on $G$ defining the quotient $G/S_z$ is right-multiplication since $\tau_z(xs)=xszs^{-1}x^{-1}=\tau_z(x)$ for all $s\in S_z$.

Since we have assumed $g,h$ are regular, we can take $S_g=T=S_h$ since all maximal tori are conjugate and the stabilizer of a regular element is a maximal torus.

Putting this together we have that each Zariski irreducible component $C\subset A\times B$ which contains a polystable homomorphism $(g,h)$ with $g,h$ regular is isomorphic to a variety of the form $G/T\times G/T$. 

Now the action defining the character variety is conjugation and that makes sense on $A \times B$ (simultaneously on $A$ and $B$).  Let us see what this action corresponds to on $G/T\times G/T$.  The isomorphism $G/T\times G/T\cong G\cdot g\times G\cdot h$ is given by $\tau:=\tau_g\times \tau_h$.  Observe that for any $a,x,y\in G$, \begin{eqnarray*}\tau(a\star (xT,yT))&=&\left(\tau_g(axT), \tau_h(ayT)\right)\\&=&(axgx^{-1}a^{-1},ayhy^{-1}a^{-1})\\&=&a\cdot(xgx^{-1},yhy^{-1})\\&=&a\cdot \tau(xT,yT),\end{eqnarray*} where $\star$ is the action of left-multiplication (in contrast to $\cdot$ which is the action of conjugation).

Therefore, we have that $C/\!\!/G$ in $\X(\Z_n*\Z_m,G)$ is isomorphic to $G\backslash(G/T \times G/T).$  We note that we do not need to take the GIT quotient since the actions of left and right multiplication are free.

For any $(xT,yT)\in G/T\times G/T$, we have  $x^{-1}\star(xT,yT)=(T,x^{-1}yT)$.  In other words, there is a representative for each point in $G\backslash(G/T \times G/T)$ of the form $[(T,yT)]$. What then remains of the left-multiplication action $\star$ on such representatives is the left-multiplication of $T$ alone since if $zT=T$ then $z\in T$.

Thus $C/\!\!/G\cong G\backslash (G/T \times G/T) \cong T\backslash G/T,$ as required.

\end{proof}

\begin{thm}
Let $n,m$ be coprime integers, and $\Gamma_{(n,m)}$ the corresponding torus knot group.  Then the closure of each irreducible component in $\X^{irr}(\Gamma_{(n,m)}, \SL(2,\C))$ is isomorphic to $\C$.
\end{thm}

\begin{proof}
For torus knot groups $\Gamma_{(n,m)}$ with $\gcd(n,m)=1$, it must be the case that any irreducible representation corresponds to a pair $(g,h)\in G^2$ such that either $g^n=I=h^m$ or $g^n=-I=h^m$. We only show the roots of unity case as the roots of negative unity case is similar.

For $G=\SL(2,\C)$, as in the case of the above lemma, we have that $\hom(\Z_n*\Z_m, G)\cong A \times B$ where $A$ is the set of $n$-th roots of the identity matrix in $G$ and $B$ is the set of $m$-th roots of the identity matrix in $G$.  We are only considering irreducible homomorphisms, and so any irreducible component in $A\times B$ is the product of conjugation orbits of fixed diagonal matrices with $n$-th and $m$-th roots of unity as eigenvalues respectively (with the assumption that no eigenvalue $\lambda$ satisfies $\lambda^2=1$ else we would not have an irreducible representation).  Consequently, both component elements $g$ and $h$ are regular.

By Lemma \ref{lem:doublecoset}, any irreducible component of $\X^{irr}(\Z_n*\Z_m, G)$ with regular component elements is isomorphic to $T\backslash G/T.$

Now take a generic matrix $\left(\begin{array}{cc}a& b\\ c& d\end{array}\right).$ Then multiplying independently by $T$ on the left and right, as one needs to do to understand the double coset space $T\backslash G/T,$ gives:
$$\left(\begin{array}{cc}\lambda \mu a& \lambda^{-1}\mu b\\ \lambda\mu^{-1} c& \lambda^{-1} \mu^{-1}d\end{array}\right).$$ 

The invariants of this action are $ad$ and $bc$ with the single relation $ad-bc-1$.

Thus, the invariant ring is $$\C[ad, bc]/\langle ad-bc-1\rangle \cong \C[x]$$ where $x=ad$.  Thus, we find that the irreducible component is $$\mathrm{Spec}_{max}(\C[x])=\C.$$ 

\end{proof}

\subsection{Irreducible Representations of Torus Knot Groups for $m=2$}

In this section, we study the locus of irreducible representations of torus knot and link groups in the simplest case $G=\SL(2,\C)$. As shown before, any polystable representation $\rho\in\hom(\Gamma_\n,G)$ corresponds
to a $r$-tuple:
\[
(A_{1},A_{2},\ldots,A_{r})\in \SL(2,\C)^{r},
\]
with $A_{i}:=\rho(\gamma_{i})$ semisimple, and $A_{i}^{n_{i}}=A_{j}^{n_{j}}$
for all $i,j$. We start by stating a necessary and sufficient condition
for $\rho$ to be irreducible.

Denote the commutator of two matrices $A,B$ by $[A,B]:=ABA^{-1}B^{-1}$.  We say the commutator is non-trivial when it is not the identity matrix.

\begin{lem}\label{lem:comm}
Suppose $\rho$ is polystable. Then, $\rho$ is irreducible if and
only if some commutator $[A_{i},A_{j}]$ is non-trivial. 
\end{lem}

\begin{proof}
In the $n=2$ case, if $\rho$ is polystable and reducible, then it is abelian and all commutators are trivial. If $\rho$ is irreducible and polystable,
then without loss of generality, we can assume $A_{1}$ diagonal and
non-scalar. Moreover, by irreducibility, some $A_{j}$ with $j\neq1$
has to be non-diagonal. Let $A_1=\left(\begin{array}{cc}x&0\\0&x^{-1}\end{array}\right)$ with $x^2\not=1$, and let $A_j=\left(\begin{array}{cc}a&b\\c&d\end{array}\right)$.  Then solving $A_1A_j=A_jA_1$ for $a,b,c,d$ gives that $b=0=c$ which contradicts $A_j$ being non-diagonal.  Thus, $[A_{1},A_{j}]$ is non-trivial.
\end{proof}

We can now characterize the irreducible components of $\hom^{irr}(\Gamma_\n,G)$
in terms of the eigenvalues of all the $A_{i}$'s. Let $\{\lambda_{i},\lambda_{i}^{-1}\}$
be the eigenvalues of $A_{i}$, and let $\rho$ be irreducible. 

Then,
by Corollary \ref{important-cor}, these values satisfy:
\begin{equation}
\lambda_{i}^{n_{i}}=1,\quad\text{for all }i;\quad\quad\text{or}\quad\quad\lambda_{i}^{n_{i}}=-1,\quad\text{for all }i.\label{eq:roots}
\end{equation}
(there is no pair of indices $i\neq j$ such that $\lambda_{i}^{n_{i}}=1$
and $\lambda_{j}^{n_{j}}=-1$), and note that $A_{i}$ is a central element,
if and only if $\lambda_{i}=\lambda_{i}^{-1}$ (that is $\lambda_{i}=\pm 1$).
Let us use the notation 
\[
D(\lambda):=\left(\begin{array}{cc}
\lambda & 0\\
0 & \lambda^{-1}
\end{array}\right)\in \SL(2,\C).
\]

\begin{lem}\label{lem:algmap}
Fix a $r$-tuple $(\lambda_{1},\lambda_{2},\ldots,\lambda_{r})$ such
that either $\lambda_{i}^{n_{i}}=1$ for all $i\in[r]$ or $\lambda_{i}^{n_{i}}=-1$ for all $i\in[r]$. Then, for $D_{r}:=D(\lambda_{r})$
the image of the algebraic map
\[\varphi:
G^{r}\to\hom(\Gamma_\n,G),\quad\quad(g_{1},\ldots,g_{r})\mapsto(g_{1}D_{1}g_{1}^{-1},\ldots,g_{r}D_{r}g_{r}^{-1})
\]
intersects $\hom^{irr}(\Gamma_\n,G)$ if and only if \emph{at least 2 of the $\lambda_{i}$'s
are different than $\pm1$}. In this case, the image is the closure
of a single irreducible component of $\hom^{irr}(\Gamma_\n,G)$.
\end{lem}

\begin{proof}
If all $\lambda_{i}=\pm1$ then the image of the map is a point, and this point corresponds to a reducible representation. Suppose, without loss of generality,
that $\lambda_{1}\neq\pm1$ but all others are $\pm1$. Then,
\[
(g_{1}D_{1}g_{1}^{-1},\ldots,g_{r}D_{r}g_{r}^{-1})=(g_{1}D_{1}g_{1}^{-1},\pm I_{2},\ldots,\pm I_{2}),
\]
and this is always a reducible representation, regardless of $g_{1}\in \SL(2,\C)$.

Finally, and again without loss of generality, suppose $\lambda_{1}$ and $\lambda_{2}$ are both different than $\pm1$. Then, there is a solution $g_{1},g_{2}\in \SL(2,\C)$ to the equation $[g_{1}D_{1}g_{1}^{-1},g_{2}D_{2}g_{2}^{-1}]\neq I_{2}$, which, by Lemma \ref{lem:comm}, means that the corresponding representation is
irreducible. 

It remains to prove, in the third case considered above, that the image of $\varphi$ is the closure of a single irreducible component in $\hom^{irr}(\Gamma_\n,G)$.

To see this note that $G^r$ is a group and the image of $\varphi$ is the $G^r$-orbit of the point $(D_1,...,D_r)$ in $\hom(\Gamma_n,G)$.

Since the $D_i$'s are diagonal, the orbit is closed and the stabilizer $H$ is reductive.  Therefore, the image of $\varphi$ is isomorphic to the affine variety $G^r/H$, the latter being irreducible since $G^r$ is an irreducible variety.

As we vary the eigenvalues defining the $D_i$'s, we see the images of $\varphi$ are disjoint sets in $\hom^{irr}(\Gamma_\n,G)$ since the different choices are discrete (roots of unity).  On the other hand, every irreducible representation whose  component matrices have the same eigenvalues as the corresponding components of $(...,D_i,...)$ must be in the image of $\varphi$.  Thus, the image of $\varphi$ is the closure (in $\hom(\Gamma_\n,G)$) of an irreducible component of $\hom^{irr}(\Gamma_\n,G)$, as required.
\end{proof}

Lemma \ref{lem:algmap} means that $r$-tuples $(\lambda_{1},\ldots,\lambda_{r})$
can be used to parametrize irreducible components of $\hom^{irr}(\Gamma_{n},\SL(2,\C))$,
as long as 2 or more eigenvalues are not $\pm1$. However, when $\sigma_{i}\in\{\pm1\}$ for
$i\in[r]$, the $r$-tuples:
\begin{equation}
(\lambda_{1},\lambda_{2},\ldots,\lambda_{r})\quad\text{and }\quad(\lambda_{1}^{\sigma_{1}},\lambda_{2}^{\sigma_{2}},\ldots,\lambda_{r}^{\sigma_{r}}),\label{eq:4classes}
\end{equation}
all parametrize the same component. This is a consequence of the fact
that $\SL(2,\C)$ is connected so that
\[
A_0:=D(\lambda)=\left(\begin{array}{cc}
\lambda & 0\\
0 & \lambda^{-1}
\end{array}\right)\quad\text{and}\quad A_1:=D(\lambda^{-1})=\left(\begin{array}{cc}
\lambda^{-1} & 0\\
0 & \lambda
\end{array}\right)
\]
are the two endpoints of a path of the form $t\mapsto A(t):=g(t)A_0g(t)^{-1}$,
$t\in[0,1]$.  By dimensional reasons, such paths can be chosen to avoid reducible representations.

In conclusion, writing $F=C_{2}^{r}$ where $C_{2}:=\{1,-1\}$ is the
multiplicative 2-element group, there is an action of $F$ on the
$r$-tuples $(\lambda_{1},\lambda_{2},\ldots,\lambda_{r})$ given
by:
\[
(\sigma_{1},\ldots,\sigma_{r})\cdot(\lambda_{1},\lambda_{2},\ldots,\lambda_{r}):=(\lambda_{1}^{\sigma_{1}},\lambda_{2}^{\sigma_{2}},\ldots,\lambda_{r}^{\sigma_{r}}).
\]
The number of orbits of this action (on the set $(...,\lambda_{i},...)$
with at least two $\lambda_{i}\neq\pm1$) gives the number of irreducible
components of $\hom^{irr}(\Gamma_\n,G)$. Note also that there are no further possible
identifications between the $r$-tuples $(...,\lambda_{i},...)$ since the
image of the map of Lemma \ref{lem:algmap} is a single algebraic irreducible
component of $\hom^{irr}(\Gamma_\n,G)$.

\begin{thm}\label{thm:components}
Suppose all $n_{i}$ in $\n$ are odd. Then, the number of irreducible components
of $\hom^{irr}(\Gamma_{\n},\SL(2,\C))$, and consequently of $\X^{irr}_{\Gamma_\n}(\SL(2,\C))$, is:
\[
r-2-\sum_{i=1}^{r}n_{i}+\frac{1}{2^{r-1}}\prod_{i=1}^{r}(n_{i}+1).
\]
\end{thm}

\begin{proof}
We use Burnside's Lemma to count orbits of $F=C_{2}^{r}$ acting
on the whole set $X=X_{+}\sqcup X_{-}$ where 
\[
X_{\pm}:=\{(\lambda_{1},\lambda_{2},\ldots,\lambda_{r})\,:\,\lambda_{i}^{n_{i}}=\pm1,\quad\text{for all }i\},
\]
and subtract the number of orbits corresponding to $$(1,\ldots,1),
(-1,\ldots,-1),$$ and $(...,\lambda_{i},...)$ with a single entry different
than $\pm1$. The number of these exceptional cases is
\[
2+\sum_{i=1}^{r}(n_{i}-1),
\]
which gives the negative of the first term in the formula. Now, $|X/F|=|X_{+}/F|+|X_{-}/F|$
and Burnside's Lemma states:
\[
|X_{+}/F|=\frac{1}{|F|}\sum_{\sigma\in F}|X_{+}^{\sigma}|
\]
where $X_{+}^{\sigma}=\{(\lambda_{1},\lambda_{2},\ldots,\lambda_{r})\in X_{+}\ |\ (\lambda_{1},\lambda_{2},\ldots,\lambda_{r}) \text{ is fixed by }\sigma\}$.
We see that if $\sigma=(\sigma_{1},\ldots,\sigma_{r})$ is such that
$1\leq i_{1}<\ldots<i_{k}\leq r$ are the indices corresponding to
the element $+1\in C_2$, then $\sigma$ fixes exactly $n_{i_{1}}\cdots n_{i_{k}}$
elements of $X_{+}$ so we end up with the sum
\[
\frac{1}{2^{r}}(n_{1}\cdots n_{r}+n_{2}\cdots n_{r}+\cdots+1)=\frac{1}{2^{r}}(n_{1}+1)\cdots(n_{r}+1).
\]
Finally, the count of orbits for $X_{-}$ gives the same number.
\end{proof}

\begin{example}
Let $(n_1,n_2)=(5,7)$ as in Figure \ref{fig:polygons}. Then 
$X_+=\{(\nu_j,\mu_k)\,:\, 0\leq j\leq 4, 0\leq k\leq 6\}$, since $\nu_j=e^{2i\pi j/5}$ are the fifth roots of 1, and similarly $\mu_k$ are the seventh roots of 1. The orbits of the action of $(C_2)^2$ on $X_+$ are $\frac14(35+7+5+1)=12$. The count of the orbits for $X_-$ (the set of pairs $(\lambda_1,\lambda_2)$ with $\lambda_i^{n_i}=-1$) also gives 12 orbits, since $(\nu_j,\mu_k)\in X_+$ if and only if $(-\nu_j,-\mu_k)\in X_-$. Finally, the exceptional orbits, to be removed, are the ones of $(1,1),(-1,-1),(\nu_j,1),(-\nu_j,-1),(1,\mu_k),(-1,-\mu_k)$, with $j,k\neq 0$ which give $2+4+6=12$ elements, given the identifications between $(1,\nu_j)$ and $(1,\nu_{5-j})$ and similarly for the other pairs. Hence, the number of irreducible components of $\hom^{irr}(\Gamma_{(5,7)},\SL(2,\C))$ is 12.
\end{example}

\begin{figure}[hbt] \centering
\includegraphics[height=32mm]{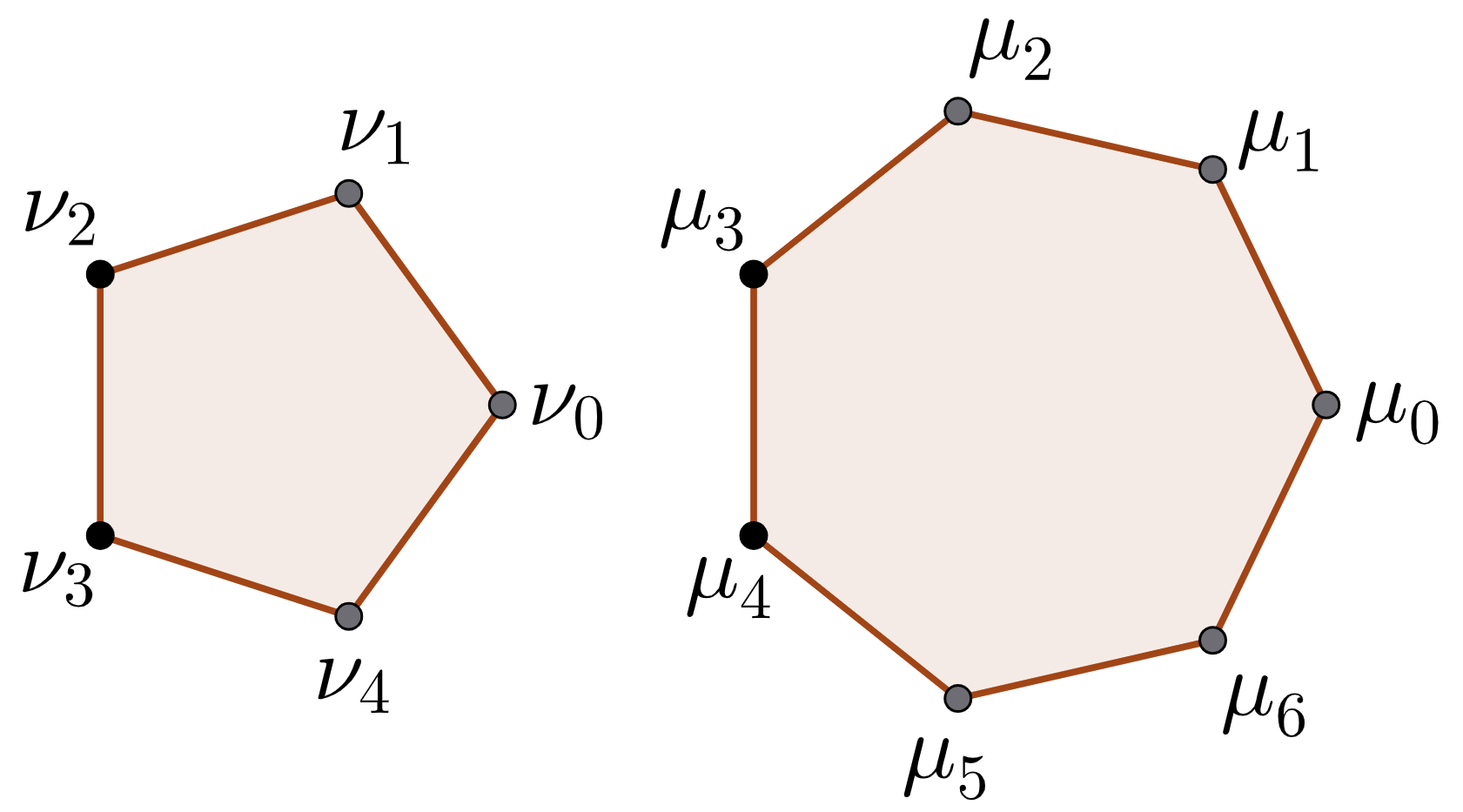}
\caption{Counting Components for $(n_1,n_2)=(5,7)$.}
\label{fig:polygons}
\end{figure}

\begin{rem}
\begin{enumerate}
\item[]
\item When $r=2$ the formula of Theorem \ref{thm:components} gives $\frac{1}{2}(n_{1}-1)(n_{2}-1)$ which matches the number found in \cite{munozsl2,ollersl2link}. Hence, Theorem \ref{thm:components} generalizes
that result to generalized torus link groups. See also Corollary \ref{cor:gl2} for the case $\GL(2,\C)$.

\item The same formula works when $n_{1}$ is even and all other $n_{i}$ are odd. In this case, the number of orbits of the action of $F=(C_2)^r$ on $X_{+}$ and on $X_{-}$ is different, but they add up to the same formula. We believe that an example is more instructive than the general argument. Indeed, if $(n_1,n_2)=(4,5)$ as in Figure \ref{fig:polygons2}, $X_+$ is the set of 20 pairs $(\nu_j,\mu_k)$ with $\nu_j$ 4th root and $\mu_k$ 5th root of 1, and $X_-$ is the set of 20 pairs $(\xi_j,-\mu_k)$ which are roots of $-1$. However, while for $X_+$, every element $(-1,\sigma_2)\in (C_2)^2$ fixes both $\nu_0=1$ and $\nu_2=-1$, the same element of $(C_2)^2$ does not fix any element of $X_-$ since $\pm1$ are not 4th roots of $-1$. Hence, Burnside's lemma gives: $|X_{+}/F|=\frac14(20+4+10+2)=9$ but only $|X_{-}/F|=\frac14(20+4+0+0)=6$. The exceptional cases, to be removed, are computed to be 9, and we end up with 6 irreducible components, matching the expression $\frac12(n_1-1)(n_2-1)$. The same techniques also work in the cases with other even $n_i$'s. All calculations can be easily performed with a computer program, but a closed formula for all cases is rather involved.
\end{enumerate}

\end{rem}

\begin{figure}[hbt] \centering
\includegraphics[height=26mm]{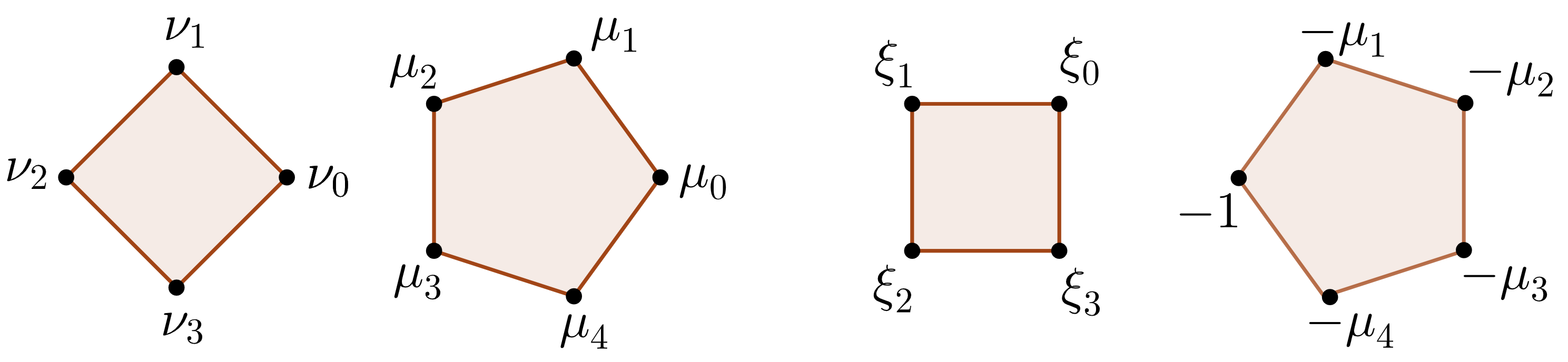}
\caption{Counting Components for $(n_1,n_2)=(4,5)$.}
\label{fig:polygons2}
\end{figure}

\subsection{Some general results for $\GL(m,\C)$}

Consider the canonical projection:
\begin{eqnarray*}
\pi:\hom^{irr}(\Gamma_{\mathbf{n}},\GL(m,\C)) & \to & \mathbb{C}^{*},\\
(A_{1},\ldots,A_{r}) & \mapsto & \omega
\end{eqnarray*}
where $\omega\in\mathbb{C}^{*}$ is such that $A_{i}^{n_{i}}=\omega I_{m}$ for all $i\in[r]$.

Note that, in particular,
\begin{equation}
\label{eq:pi}
\hom^{irr}(*\Z_{\mathbf{n}},\GL(m,\C))=\pi^{-1}(1),
\end{equation}
where  $*\Z_{\mathbf{n}}:=\mathbb{Z}_{n_{1}}*\mathbb{Z}_{n_{2}}*\cdots*\mathbb{Z}_{n_{r}}$
is the free product of cyclic groups of orders $n_{i}$. 

Given $\rho=(A_1,...,A_r)\in\pi^{-1}(\omega)$ and a fixed $i\in[r]$, all eigenvalues of $A_i$ are $n_i$th-roots of $\omega$.

We will now consider a subvariety of $\hom(\Gamma_{\mathbf{n}},\GL(m,\C))$
consisting of those representations such that every $A_{i}\in\GL(m,\C)$
has exactly $m$ distinct eigenvalues (every $A_{i}$ is a regular semisimple element).
\begin{defn}Let $\Gamma$ be a group generated by $\gamma_1,...,\gamma_r$. We define $\mathcal{R}_{\Gamma,m}^{de}\subset\hom(\Gamma,\GL(m,\C))$
as:
\[
\mathcal{R}_{\Gamma,m}^{de}:=\{\rho\in\hom(\Gamma,\GL(m,\C))\mid\rho(\gamma_{i})\text{ has } m \text{ distinct eigenvalues for all } i\in [r]\}.
\]
\end{defn}

We are interested in determining the number of path components of
$\mathcal{R}_{\Gamma_{\mathbf{n}},m}^{de}$.  For this, we can assume,
without loss of generality, that the $n_{i}$ are ordered in non-decreasing
order:
\[
n_{1}\leq n_{2}\leq\cdots\leq n_{r}.
\]

By \eqref{eq:pi} we have
\[
\mathcal{R}_{*\Z_{\mathbf{n}},m}^{de, irred}:=\mathcal{R}_{\Gamma_{\mathbf{n}},m}^{de}\cap\pi^{-1}(1)\subset\hom^{irr}(*\Z_{\mathbf{n}},\GL(m,\C)).
\]
As before, all eigenvalues of the matrix $A_{i}$ in $\rho\in\mathcal{R}_{*\Z_{\mathbf{n}},m}^{de}$
are $n_{i}$-th roots of unity. Hence, if $m>n_{1}$ then $\mathcal{R}_{*\Z_{\mathbf{n}},m}^{de}$
is empty (the same reasoning applies to $\mathcal{R}_{\Gamma_{\mathbf{n}},m}^{de}$
for which the eigenvalues of the $A_{i}$ are $n_{i}$-th roots of
$\omega$).

\begin{lem}
\label{lem:Z_n}Let $m\leq n_{1}$. Then, the number of path connected
components of $\mathcal{R}_{*\Z_{\mathbf{n}},m}^{de, irred}$ is
\[
N_{m}(n):=\prod_{i=1}^{r}\binom{n_i}{m}.
\]
\end{lem}

\begin{proof}

First we note that there is no path between different $n_{i}$-th roots of unity (preserving the relation), so it suffices to count such choices to determine the number of path-components.

The hypothesis implies that $A_{i}^{n_{i}}=I_{m}$ and all eigenvalues
of $A_{i}$ are pairwise distinct $n_{i}$-th roots of unity. So, for
each $A_{i}$, we have $\binom{n_i}{m}$ possible choices of such
eigenvalues (up to reordering). The ordering of these eigenvalues is irrelevant because
$A_{i}$ can be conjugated by a permutation matrix $P\in\GL(m,\C)$, keeping the other matrices fixed, and there is a path $\gamma:[0,1]\to\GL(m,\C)$
joining $A_{i}$ with $PA_{i}P^{-1}$. Such a path gives rise to a
path
\[
\tilde{\gamma}:[0,1]\to\mathcal{R}_{*\Z_{\mathbf{n}},m}^{de},\quad\quad\tilde{\gamma}(t)=(A_{1},\ldots,A_{i-1},\gamma(t),A_{i+1},\ldots,A_r).
\]
Using a codimension argument, one can show that $\tilde{\gamma}(t)$
can always be deformed so that it lies inside the open locus of irreducible
representations in $\hom(*\Z_{\mathbf{n}},\GL(m,\C))$.

We note a technical fact:  the (deformed) path above may result (at time $t=1$) in a representation that is {\it reducible}.  This does not change the count of components however since the closure, taken in $\hom^{de}(*\Z_{\mathbf{n}},\GL(m,\C))$, will have the same number of components as $\mathcal{R}_{*\Z_{\mathbf{n}},m}^{de, irred}$ since (as mentioned before) there can be no path between tuples with different $n_{i}$-th roots of unity.
\end{proof}

Now, to determine the number of path connected components of $\mathcal{R}_{\Gamma_{\mathbf{n}},m}^{de}$,
we first define a path in the subgroup
\[
\GL(m,\C)^{1}:=\{A\in\GL(m,\C)\,:\,|\det A|=1\}
\]
as follows. Fix $A\in\GL(m,\C)^{1}$ and $n$ and let:
\begin{eqnarray*}
\gamma_{A,n}:\mathbb{R} & \to & \GL(m,\C)^{1}\\
t & \mapsto & A\,e^{2\pi it\frac{1}{n}},
\end{eqnarray*}
whose image is a compact curve. If, in particular, $A^{n}=I_{m}$,
then $(\gamma_{A}(k))^{n}=I_{m}$, for all $k\in\mathbb{Z}$ (moreover,
$\gamma_{A,n}(t)$ has the same eigenspaces as $A$ since a matrix shares
the same eigenspaces with a non-zero scalar multiple of it).

It is then clear that $\rho=(A_{1},\ldots,A_{i},\ldots,A_{r})$ and
$$
\rho_{i}(t):=(\gamma_{A_{1},n_1}(t),\ldots,\gamma_{A_{i},n_i}(t),\ldots,\gamma_{A_{r},n_r}(t))
$$ are
in the same path component of $\mathcal{R}_{\Gamma_{\mathbf{n}},m}^{de}$
(although generally in distinct path components of $\mathcal{R}_{*\Z_{\mathbf{n}},m}^{de}$).
So, we define, for $k\in\mathbb{Z}$,
\[
\varphi(k,\rho)=(\gamma_{A_{1},n_1}(k),\ldots,\gamma_{A_{r},n_r}(k))
\]
and this is clearly an action of $\mathbb{Z}$ on $\mathcal{R}_{*\Z_{\mathbf{n}},m}^{de}$
(this actually factors through an action of $\mathbb{Z}_{N}$, where $N$ is the least common multiple of all the $n_i$).
Thus, $\mathbb{Z}$ also acts on the set $\pi_{0}(\mathcal{R}_{*\Z_{\mathbf{n}},m}^{de})$
of path components of $\mathcal{R}_{*\Z_{\mathbf{n}},m}^{de}$. So,
we have shown:
\begin{prop}
The number of path components of $\mathcal{R}_{\Gamma_{\mathbf{n}},m}^{de}$
is
\[
\left|\pi_{0}(\mathcal{R}_{*\Z_{\mathbf{n}},m}^{de})/\mathbb{Z}\right|,
\]
that is, the $($finite$)$ number of $\mathbb{Z}$-orbits on $\pi_{0}(\mathcal{R}_{*\Z_{\mathbf{n}},m}^{de})$.
\end{prop}

\begin{proof}
Here we use the fact that there is a SDR between $\mathcal{R}_{\Gamma_{\mathbf{n}},m}^{de}$
and the subspace where $\rho$ has all matrices $A_{i}$ with determinant
in the circle $\U(1)$ by Lemma \ref{lem:def-S1}.
\end{proof}

\begin{cor}
\label{cor:gl2}
Let $m=r=2$, and $\gcd(n_{1},n_{2})=1$. Then
\[
|\pi_{0}(\hom^{irr}(\Gamma_{\mathbf{n}},\GL(m,\C)))|=\left\lfloor \frac{n_{1}}{2}\right\rfloor \left\lfloor \frac{n_{2}}{2}\right\rfloor
\]
\end{cor}

\begin{proof}
Since $m=r=2$ we have $\mathcal{R}_{\Gamma_{\mathbf{n}},m}^{de, irred}=\hom^{irr}(\Gamma_{\mathbf{n}},\GL(m,\C))$.
In fact, if one of the two matrices $A_{1},A_{2}$ is a central element, then
$(A_{1},A_{2})$ is not an irreducible pair. Hence, the result follows from
computing directly the number of distinct configurations of unordered
choices of two distinct $n_{1}$-th and two distinct $n_{2}$-th roots of
1, up to the action of $\mathbb{Z}$ described above. We note that the number of components matches the count performed in \cite[Proposition 7.3]{munozsl3}.
\end{proof}

\begin{rem}
Since $\X(\Gamma_\n,G)$ and $\hom(\Gamma_\n,G)$ have the same number of components $($since $G$ is connected$)$, and $$\hom(\Gamma_\n,\GL(m,\C))/(\C^*)^r\cong \hom(\Gamma_\n,\mathrm{PGL}(m,\C)),$$ we have that $\pi_0(\X(\Gamma_\n,\GL(m,\C)))=\pi_0(\X(\Gamma_\n,\mathrm{PGL}(m,\C)))$.  Thus the previous corollary applies to $G=\mathrm{PGL}(m,\C)$ as well since the action of $(\C^*)^r$ preserves the irreducible locus in $\hom(\Gamma_\n,\GL(m,\C))$.
\end{rem}

\appendix
\section{Roots in a Lie Group}\label{appendix:roots}

In this appendix we consider the characterization of spaces of $n$-th roots inside a general Lie group $G$. For a given $x\in G$ let $\sqrt[n]{x}^{G}$
denote the set of $n$-th roots of $x\in G$ by
\[
\sqrt[n]{x}^{G}:=\{y\in G\,|\,y^{n}=x\}.
\]
In general, $\sqrt[n]{x}^{G}$ may be infinite, but it can always
be decomposed into disjoint classes under conjugation. 

Let us also denote by $C^{n}(x)^{G}$ the set of conjugacy classes
of elements in $\sqrt[n]{x}^{G}$, and by $|C^{n}(x)^{G}|$ its cardinality
(when finite). In McCrudden \cite{Mc} (Corollary to Proposition 2) the
following result is proved. Let $ZG$ and $DG$ denote, respectively,
the center and the derived group of $G$.
\begin{prop}
\label{prop:McC-bound} Let $e\in G$ be the identity. Then, for all
$x\in G$, we have: 
\[
|C^{n}(x)^{G}|\leq|\sqrt[n]{e}^{Z(Z_{x})}|\cdot|C^{n}(e)^{D(Z_{x})}|,\text{ for all } x\in G,
\]
where $Z_{x}\subset G$ denotes the centralizer of $x$ in $G$, provided
both factors are finite.
\end{prop}

From this, we conclude:
\begin{thm}\label{thm:conj-class}
Fix $n\in\mathbb{N}$. Let $G$ be a connected
reductive group, with connected center $ZG$, and let $x\in ZG$.
Then:
\begin{enumerate}
\item $C^{n}(x)^{G}$ is finite. 
\item Every such conjugacy class is isomorphic to a homogeneous space of
the form $G/Z_{y}$, for a given $y\in\sqrt[n]{x}^{G}$, where $Z_{y}$
is the centralizer of $y$ in $G$.
\end{enumerate}
\end{thm}
\begin{proof}
(1) In McCrudden \cite{Mc} (Corollary 2 of Proposition 4), it is shown that if a group $G$ is connected,
then $C^{n}(e)^{G}$ is finite. Our hypothesis on $G$ implies that $DG$ is connected. Indeed,
$DG$ is the continuous image of the commutator map from $G\times G$
which is connected. Then, since $ZG$ is also connected

\[
\sqrt[n]{e}^{ZG}=C^{n}(e)^{ZG},\quad\mbox{and }C^{n}(e)^{DG}
\]
are both finite (a set of conjugacy classes in an abelian group is
itself). Now, applying Proposition \ref{prop:McC-bound} above, for
$x\in ZG$, we have $Z_{x}=G$ and we obtain: 
\[
|C^{n}(x)^{G}|\leq|\sqrt[n]{e}^{ZG}|\cdot|C^{n}(e)^{DG}|,
\]
which proves (1). 

(2) It is clear that, given $y\in\sqrt[n]{x}^{G}$, for $x\in ZG$,
we have:
\[
(gyg^{-1})^{n}=gy^{n}g^{-1}=gxg^{-1}=x,
\]
so that $\{gyg^{-1}:g\in G\}$ is a conjugacy class in $C^{n}(x)^{G}$.
But every conjugacy class is of this form. Finally, since there is
an transitive action of $G$ on this conjugacy class and the subgroup
$Z_{y}\subset G$ is the stabilizer of a point, the orbit-stabilizer
theorem implies that this conjugacy is isomorphic to $G/Z_{y}$.\end{proof}

Given a set $S\subset G$, denote the collection of conjugation orbits of elements of $S$ by $GSG^{-1}:=\{gsg^{-1}\ |\ s\in S\}$.

\begin{lem}
\label{lem:T-n}For $G=\GL(m,\C)$, any matrix in $\sqrt[n]{I}^{G}$
belongs to $GT_{n}G^{-1}$ where $T_{n}$ is the group of diagonal
matrices with $n$-th roots of unity in the diagonal.\end{lem}
\begin{proof}
Since $y\in\sqrt[n]{I}^{G}$ is diagonalizable, there is $g\in G$
such that $gyg^{-1}$ is diagonal, say with diagonal entries $\lambda_{1},\ldots,\lambda_{m}$.
Hence
\[
1=gg^{-1}=gy^{n}g^{-1}=(gyg^{-1})^{n},
\]
has diagonal entries $\lambda_{1}^{n}=\cdots=\lambda_{m}^{n}=1$.
So, this diagonal matrix is in $T_{n}$. 
\end{proof}
This lemma serves to label conjugacy classes, and actually gives their precise
number.
\begin{prop}
The set $\sqrt[n]{I}$ has exactly $\binom{m+n-1}{n-1}$ distinct
conjugacy classes.\end{prop}
\begin{proof}
For this count, from Lemma \ref{lem:T-n}, we only need to consider
the conjugacy classes of elements in $T_{n}\cong\mathbb{Z}_{n}\times\cdots\times\mathbb{Z}_{n}$
which is a product of cyclic groups of order $n^{m}$. However, some
distinct elements in $T_{n}$ may belong to the same conjugacy class.
This class is actually determined by the number of equal eigenvalues,
with their order being irrelevant. So, this counting problem is equivalent
to the distribution of (indistinguishable) $m$ balls (the number
of eigenvalues, which is the size of the matrix) into $n$ slots (the
elements in $\mathbb{Z}_{n}$, which is isomorphic to the group of
$n$-th roots of unity in $\mathbb{C}$), which reduces to the given
binomial number.
\end{proof}

\def\cdprime{$''$} \def\Dbar{\leavevmode\lower.6ex\hbox to 0pt{\hskip-.23ex
  \accent"16\hss}D}

\end{document}